\documentclass[a4paper, 11pt]{amsart}
\usepackage{amsfonts}
\usepackage{tensor}
\usepackage{mathrsfs}
\usepackage {amssymb}
\usepackage {amsmath}
\usepackage {amsmath}
\usepackage {amsthm}
\usepackage{graphicx}
\usepackage{tikz}
\usepackage {amscd}
\usepackage[colorlinks,linkcolor=black, anchorcolor=black, citecolor=red]{hyperref}
\newcommand{\sddbar}{{\sqrt{-1}\partial\bar{\partial}}}

\newcommand{\mY}{{\mathcal{Y}}}
\newcommand{\mL}{{\mathcal{L}}}
\newcommand{\mX}{{\mathcal{X}}}
\newcommand{\mD}{{\mathcal{D}}}
\newcommand{\mM}{{\mathcal{M}}}

\newcommand{\mH}{{\mathcal{H}}}
\newcommand{\C}{\mathbb C}
\newcommand{\R}{\mathbb R}

\newcommand{\Q}{\mathbb Q}
\newcommand{\cP}{\mathbb P}

\newtheorem{thm}{Theorem}[section]
\newtheorem{prop}[thm]{Proposition}
\newtheorem{defn}[thm]{Definition}

\newtheorem{cor}[thm]{Corollary}
\newtheorem{rem}[thm]{Remark}
\newtheorem{conj}[thm]{Conjecture}
\newtheorem{exmp}[thm]{Example}
\newtheorem{lem}[thm]{Lemma}
\newtheorem{claim}[thm]{Claim}

\begin{document}

\title{Conical K\"{a}hler-Einstein metric revisited}
\author{Chi Li, Song Sun}
\date{\today}
\maketitle{}

\begin{abstract}
In this paper we introduce the ``interpolation-degneration" strategy
to study K\"ahler-Einstein metrics on a smooth Fano manifold with
cone singularities along a smooth divisor that is proportional to
the anti-canonical divisor. By ``interpolation" we show the angles
in $(0, 2\pi]$ that admit a conical K\"ahler-Einstein metric form an
interval; and by ``degeneration"  we figure out the boundary of the
interval. As a first application, we show that there exists a
K\"ahler-Einstein metric on $\cP^2$ with cone singularity along a
smooth conic (degree 2) curve if and only if the angle is in
$(\pi/2, 2\pi]$. When the angle is $2\pi/3$ this proves the
existence of a Sasaki-Einstein metric on the link of a three
dimensional $A_2$ singularity, and thus answers a problem posed by
Gauntlett-Martelli-Sparks-Yau. As a second application we prove a
version of Donaldson's conjecture about conical K\"{a}hler-Einstein
metrics in the toric case using Song-Wang's recent existence result
of toric invariant conical K\"{a}hler-Einstein metrics.
\end{abstract}

\tableofcontents

\section{Introduction}

The existence of  K\"{a}hler-Einstein metrics on a smooth K\"{a}hler
manifold $X$ is a main problem in K\"{a}hler geometry. When the
first Chern class of $X$ is negative, this was solved by Aubin
\cite{Au76} and Yau \cite{Yau1}. When the first Chern class is zero,
this was settled by Yau \cite{Yau1}. The main interest at present
lies in the case of Fano manifolds, when the first Chern class is
positive. This is the famous Yau-Tian-Donaldson program which
relates the existence of K\"ahler-Einstein metrics to
algebro-geometric stability.

More generally one could look at a pair $(X,D)$ where $D$ is a smooth divisor in a K\"ahler manifold $X$, and study the existence of K\"ahler-Einstein metrics on $X$ with cone singularities along $D$.  This problem
was classically studied on the Riemann surfaces
\cite{McO,Troy,LuTi} and was first considered in higher dimensions
by Tian in \cite{Ti94a}. Recently, there is a reviving interest on
this generalized problem, mainly due to Donaldson's program(see \cite{Do11}) on constructing smooth K\"ahler-Einstein metrics on $X$ by varying the angle along an anti-canonical divisor. There are many subsequent works, see \cite{
Berm, JMRL, Li11}.

From now on in this paper, we assume $X$ is a smooth Fano manifold,
and $D$ is a smooth divisor which is $\Q$-linearly equivalent to
$-\lambda K_X$ for some $\lambda \in\mathbb{Q}$. $\beta$ will always
be a number in $(0,1]$. We say $(X, D)$ is \emph{log canonical
(resp.log Calabi-Yau, resp. log $\mathbb{Q}$-Fano) polarized} if
$\lambda>1$ (resp. $\lambda=1$, resp. $\lambda<1$). We will study
K\"{a}hler-Einstein metrics in  $2\pi c_1(X)$ with cone
singularities along $D$.  The equation is given by
\begin{align}
Ric(\omega)=r(\beta)\omega+2\pi(1-\beta)\{D\},\tag*{$(*)$}\label{maineq}
\end{align}
where $2\pi\beta$ is the angle along $D$. For brevity we say
$\omega$ is a \emph{conical K\"ahler-Einstein metric} on $(X,
(1-\beta) D)$. Recall that the Ricci curvature form of a K\"{a}hler
metric $\omega$ is defined to be
\[
Ric(\omega)=-\sddbar\log\omega^n.
\]
In other words, the volume form $\omega^n$ determines a Hermitian
metric on $K_X^{-1}$ whose Chern curvature is the Ricci curvature.
So in particular, it represents the cohomology class $2\pi c_1(X)$.
By taking cohomological class on both sides of the equation
\ref{maineq}, we obtain
\begin{equation}\label{ricrel}
r(\beta)=1-(1-\beta)\lambda.
\end{equation}
We will use the above notation throughout this paper. Given a pair $(X,D)$, we define the set
$$E(X,D)=\{\beta\in (0,1)|\ \text{There is a conical K\"ahler-Einstein metric on}\ (X, (1-\beta)D)\}.$$

\begin{thm}\label{main1}
If $\lambda\geq 1$,  then $E(X,D)$ is a relatively open interval in $(0,1]$, which contains $(0, 1-\lambda^{-1}+\epsilon)$ for some $\epsilon=\epsilon(\lambda)>0$.
\end{thm}

The last property follows from the work of \cite{JMRL} and
\cite{Berm}, and we will recall it in Section \ref{prel}. Now
suppose $X$ admits a smooth K\"ahler-Einstein metric and $\lambda\ge
1$, then there exists a K\"ahler-Einstein metric $\omega_\beta$ on
$(X, (1-\beta)D)$ for any $\beta\in (0,1]$. By \cite{Bern3} and
\cite{Berm} we know $\omega_\beta$ is unique for $\beta\in (0,1)$.
Moreover, by the implicit function theorem in \cite{Do11}
$\omega_\beta$ varies continuously when $\beta$ varies. When $\beta$
goes to one, we have
\begin{cor}\label{maincor}
If $X$ admits a K\"ahler-Einstein metric and $Aut(X)$ is discrete, then the potential of $\omega_\beta$ converges to the potential of $\omega_{KE}$ in the $C^0$ norm, where $\omega_{KE}$ is the unique smooth K\"ahler-Einstein metric  on $X$.
\end{cor}
\begin{rem}
This is in a similar flavor as Perelman's theorem \cite{Tian-Zhu} that the K\"ahler-Ricci flow converges on a K\"ahler-Einstein Fano manifold. In particular, when $\lambda=1$ this provides evidence for Donaldson's program. An algebro-geometric
counterpart about K-stability was shown in \cite{Su11}, \cite{OS11}. When $\beta$ tends to zero, this is related to a conjecture of Tian \cite{Ti94a} that the rescaled limit should be a complete Calabi-Yau metric on the complement of $D$.
\end{rem}

When $Aut(X)$ is not discrete, we will prove the convergence of
$\omega$ to a distinguished K\"ahler-Einstein metric
$\omega^D_{KE}$, modulo one technical point, see Section
\ref{convke}. The point is that, since we need to work in different
function space corresponding to different cone angles, the
application of implicit function theorem is more delicate as shown
by Donaldson in \cite{Do11}, and Donaldson's linear theory does not
provide uniform estimate when $\beta$ is close to $1$. However, in
this case even though the K\"ahler-Einstein metrics on $X$ are not
unique, we could still identify the correct limit
K\"{a}hler-Einstein metric in the moduli space. To do this, we use
Bando-Mabuchi's bifurcation method. The result we found is that, the
only obstruction for solving the conical K\"{a}hler-Einstein metric
from $\beta=1$ to $\beta=1-\epsilon$ (for $0<\epsilon\ll 1$) comes from
the holomorphic vector fields on $X$ tangent to $D$, i.e. $Lie
Aut(X,D)$. If we assume $\lambda\ge 1$, then $Aut(X,D)$ is discrete,
so the obstruction vanishes. For more details,
see the discussion in Section \ref{convke}.\\

Another motivation for this paper come from the study of conical K\"ahler-Einstein metrics on our favorite example $\cP^2$. In this case when $D$ is a smooth curve of degree bigger than two, we are in the setting of the above theorem and we know the conical K\"ahler-Einstein metrics exist on $(X, (1-\beta)D)$ for all $\beta\in (0,1]$. When the degree is one or two, we are in the case $\lambda<1$. We have an obstruction coming from log K-stability
\begin{thm} \label{nonexistence small angle}
If $\lambda<1$, then there is no conical K\"ahler-Einstein metric on $(X, (1-\beta)D)$ for $\beta<(\lambda^{-1}-1)/n$, where $n$ is the dimension of $X$.
\end{thm}

This immediately implies that there is no K\"ahler-Einstein metric on $\cP^2$ which bends along a line, which could also be seen from the Futaki invariant obstruction. The most interesting case is

\begin{thm} \label{conic}
When $D$ is a conic in $\cP^2$, i.e. a smooth degree two curve, then $E(X,D)=(1/4, 1]$.
\end{thm}
From the proof we also speculate the limit of the conical K\"ahler-Einstein metrics $\omega_\beta$ as $\beta$ tends to $1/4$.
As an application of the above theorem, we have
\begin{cor}\label{CYcone}
A three dimensional $A_2$ singularity $x_1^2+x_2^2+x_3^2+x_4^3=0$
admits a Calabi-Yau cone metric with the natural Reeb vector field.
\end{cor}
This settles a question in \cite{GMSY}, and as mentioned in \cite{GMSY}, this might be dual to an exotic type of field theory since the corresponding
Calabi-Yau cone does not admit a crepant resolution. Note this shows that the classification of cohomogeneity one Sasaki-Einstein manifolds given in
\cite{Conti} is incomplete, which is confirmed by the numerical result and calculations by the first author in \cite{Li12b}. See Remark \ref{numerical}.

Now we briefly discuss the strategy to prove the above results. The
proof of Theorem \ref{main1} follows from the following
``interpolation" result. One point is that the log-Mabuchi-energy is
well defined on the space of {\it{admissible functions }} denoted by
$\hat{\mH}(\omega)$, which includes all the K\"{a}hler potentials of
conical K\"{a}hler metrics for different angles. The definition of
log-Mabuchi-energy and log-Ding-energy as well as
$\hat{\mH}(\omega)$ will be given in Section \ref{prel}.

\begin{prop}\label{interpolate}
As functionals on $\hat{\mH}(\omega)$, the log-Mabuchi-energy
$\mM_{\omega,(1-\beta)D}$ is linear in $\beta$. The normalized
log-Ding-energy $r(\beta) F_{\omega,(1-\beta)D}$ is concave downward
in $\beta$ up to a bounded constant. As a consequence, If the
log-Mabuchi-energy (resp. log-Ding-energy) is proper for $\beta_1\in
(0,1]$ and bounded from below for $\beta_2\in (0,1]$, then for any
$\beta$ between $\beta_1$ and $\beta_2$, the log-Mabuchi-energy
(resp. log-Ding-energy) is proper, so there exists a conical
K\"{a}hler-Einstein metric on $(X,(1-\beta)D)$.
\end{prop}
By combining Proposition \ref{interpolate} with the openness result of Donaldson \cite{Do11}, and the result of Berman \cite{Berm}(see section 4.3) we easily see that
\begin{cor}\label{conicproper}
If $\lambda\ge 1$ and there is a conical K\"ahler-Einstein metric on $(X, (1-\beta)D)$ for $0<\beta<1$, then the log-Mabuchi energy $\mM_{\omega, (1-\beta)D}$ is proper.
\end{cor}

Theorem \ref{main1} easily follows from the above proposition. In general to apply Proposition \ref{interpolate} we often need to get the lower bound of log-Mabuchi-energy. For this we introduce the ``degeneration" method. We have

\begin{thm}\label{specialdeg}
If there exists a special degeneration $(\mX, (1-\beta)\mathcal{D},
\mL)$ of $(X,(1-\beta)D)$ to a conical K\"ahler-Einstein variety
$(\mX_0, (1-\beta)\mD_0)$. Assume $\mX_0$ has isolated $\mathbb{Q}$-Gorenstein singularities. Then
the log-Ding-functional and log-Mabuchi-energy of $(X,(1-\beta)D)$
are bounded from below.
\end{thm}
\begin{rem}
Here the assumption that $\mX_0$ has isolated singularities is
purely technical, but it is satisfied for our main application
here to  prove Theorem \ref{conic}.  One would certainly expect a general
statement to be true(see Conjecture \ref{conj1}  in Section 5).
\end{rem}
In particular, we provide an alternative of a special case of a theorem of Chen \cite{Chen}:

\begin{cor}[Chen's theorem in the K\"{a}hler-Einstein case]
If there exists a special degeneration of Fano manifold $(X,J)$ to a
K\"{a}hler-Einstein manifold $(X_0,J_0)$, then the Mabuchi energy on
$X$ in the class $c_1(X)$ is bounded from below.
\end{cor}

To prove Theorem \ref{nonexistence small angle}, we need to generalize the K stability obstructions to the conic setting.

\begin{thm}\label{funstable}
If the log-Ding-functional $F_{\omega,(1-\beta)D}$ or the
log-Mabuchi-energy $\mathcal{M}_{\omega,(1-\beta)D}$ is
bounded from below(resp. proper), then the polarized pair
$((X,(1-\beta)D), -K_X)$ is log-K-semistable(resp. log-K-stable).
\end{thm}
\begin{cor}\label{kestable}
\begin{enumerate}
\item
If there exists a conical K\"{a}hler-Einstein metric on
$(X,(1-\beta)D)$, then $((X,(1-\beta)D), -K_X)$ is log-K-semistable.
As a consequence, if $\lambda\geq 1$, then $((X,(1-\beta)D),-K_X)$ is
log-K-semistable for $0\le \beta<1-\lambda^{-1}+\epsilon$ for some $\epsilon>0$
\item
Assume $\lambda\ge 1$ and $0<\beta<1$. If there
exists a conical K\"{a}hler-Einstein metric on $(X,(1-\beta)D)$, then
$((X,(1-\beta)D),-K_X)$ is log-K-stable.
\end{enumerate}
\end{cor}

Theorem \ref{conic} is proved by the above ``interpolation-degeneration" method. We first use Theorem \ref{nonexistence small angle} to show $E(X, D)\subset [1/4,1]$. Then we find an explicit special degeneration of $(X, 3/4 D)$ to $(\cP(1,1,4), 3/4D_0)$ where $D_0=\{z_3=0\}$ which admits the obvious conical K\"ahler-Einstein metric. Since $X$ itself admits a K\"ahler-Einstein metric, Theorem \ref{conic} follows from the interpolation. A technical point is that we do not get the full properness of Ding functional due to presence of holomorphic vector fields. For details, see Section \ref{P2case}.

The organization of the paper is as follows. In section \ref{prel},
we review some preliminary materials, including definition of H\"older norms with respect to conical K\"ahler metrics, various energy functionals, existence theory for conical K\"ahler-Einstein metrics. We prove Proposition \ref{interpolate} in Section \ref{energy functionals}, and prove Theorem \ref{main1} and Corollary \ref{conicproper} in Section \ref{small angles}.
In Section \ref{obstruction}, we explain the
obstructions to the existence of conical K\"{a}hler-Einstein
metrics. In particular, we prove Theorem \ref{funstable} and its
corollary \ref{kestable}, and Theorem
\ref{nonexistence small angle}. In Section \ref{degKE}, we prove Theorem \ref{specialdeg}.
In Section \ref{P2case}, we prove Theorem
\ref{conic} and obtain Corollary \ref{CYcone}. In Section  \ref{brcover}, we discuss the construction of smooth K\"{a}hler-Einstein metrics using branch
covers.
 In section \ref{convke}, we prove Corollary \ref{maincor}, and prove the general convergence modulo one technical point. This
is done by carrying out bifurcation analysis first used by
Bando-Mabuchi.

After finishing the draft of this paper, we received the paper by
Jiang Song and Xiaowei Wang \cite{SW}. In the last section
\ref{relSW}, we discuss the relation of their work to our paper.
Their results have some overlaps with ours. But mostly importantly
from our point of view, they proved an existence result in the toric
case. The conical  K\"{a}hler-Einstein spaces they obtained can serve
as the degeneration limits of toric Fano manifolds with some smooth
pluri-anticanonical divisors. So combining their existence result in
the toric case with the strategy in this paper, we show, in the
toric case, a version of Donaldson's conjecture which relates the
maximal cone angle and the greatest lower bound of Ricci curvature.
To state this result, first define
\begin{equation}\label{R(X)}
R(X)=\sup\{t|\exists\ \omega\in 2\pi c_1(X) \mbox{ such that }
Ric(\omega)\ge t\omega \}.
\end{equation}
\begin{prop}\label{Doconj}
For each $\lambda$ sufficiently divisible, there exists a sub-linear
system $\mathscr{L}_\lambda$ of $|-\lambda K_X|$ such that for any
general member $D\in\mathscr{L}_\lambda$, if $D$ is smooth, then
there exists a conical K\"{a}hler-Einstein metric on
$(X,(1-\gamma)\lambda^{-1}D)$ if and only if $\gamma\in (0,R(X))$.
\end{prop}
\begin{rem}
The smoothness assumption is easily satisfied when $\rm{dim}(X)\le
2$. It seems to be guaranteed by choosing $\mathscr{L}_\lambda$ more
carefully. See the discussion in Remark \ref{smooth}. In general, if
$D$ is not smooth, then there exists a weak solution (i.e. bounded
solution) to the conical K\"{a}hler-Einstein equation.
\end{rem}
The idea of the proof is similar to the proof of Theorem
\ref{conic}, again using the ``interpolation-degeneration" method. The sublinear system $\mathscr{L}_\lambda$ we
construct has the property that for each general member $D$ in
$\mathscr{L}_\lambda$, $(X,D)$ has a degeneration to the toric conic
K\"{a}hler-Einstein pair $(X,D_0)$ constructed by Song-Wang.

The interpolation properties of energy functionals obtained in this
paper seem to be known to some other experts in the field too. In
particular, we were informed by Professor Arezzo that he also
observed this.\\

\textbf{Acknowledgement}: We are indebted to Dr. H-J. Hein for
pointing out Corollary \ref{CYcone}, and related interesting
discussions. We are grateful to Professor Jian Song and Xiaowei Wang
for sending their paper to us. Their existence result in the toric
case is a main impetus for us to write the last section in this
paper. We also thank them for pointing out our first incorrect
argument showing that generic divisor appearing in the proof of
Proposition \ref{Doconj} is smooth.  We thank Professor Xiuxiong
Chen, Professor S.K. Donaldson, Professor Gang Tian for encouraging
discussions. The second author would like to thank Yanir Rubinstein
for helpful discussions. Part of this work was done during Chi Li's
visit to Imperial College and University of Cambridge in UK.  He
would like to thank Professor S.K.Donaldson and Professor Julius
Ross for the invitation which made this joint work possible.

\section{Existence theory on conical K\"ahler-Einstein metrics}\label{prel}

\subsection{Space of admissible potentials}\label{conicspace}

In this paper, all K\"ahler metrics would be in the first Chern class $c_1(X)$.
\begin{defn}
\begin{enumerate}
\item
A conical K\"{a}hler metric on $(X,(1-\beta)D)$ is a K\"{a}hler
current $\omega$ in the class $c_1(X)$ with locally bounded potential, smooth on $X\setminus D$, and for any point $p\in D$, there is a local coordinate $\{z_i\}$ in a
neighborhood of $p$ such that $D=\{z_1=0\}$ such that  $\omega$ is
quasi-isometric to the model metric:
\[
\frac{dz_1\wedge d\bar{z}_1}{|z_1|^{2(1-\beta)}}+\sum_{i=2}^n
dz_i\wedge d\bar{z}_i.
\]
Geometrically, $\omega$ represents a K\"{a}hler metric with cone singularities along $D$ of  angle $2\pi\beta$.
\item
A conical K\"{a}hler-Einstein metric on $(X,(1-\beta)D)$ is a conical
K\"{a}hler metric  solving the  equation
\begin{equation*}
Ric(\omega)=r(\beta)\omega+2\pi(1-\beta)\{D\}.
\end{equation*}
Here $\{D\}$ is the current of integration on $D$, and $
r(\beta)=1-(1-\beta)\lambda$.
\end{enumerate}
 \end{defn}
 Now we follow Donaldson \cite{Do11} to defined the H\"older norm with respect to conical metric.
 let $(z_1, z_2,\cdots, z_n)$ be the coordinates near a point in $D$ as chosen above. Let $z=re^{i\theta}$ and let  $\rho=r^{\beta}$. The model metric is
 \begin{eqnarray*}
\omega=
(d\rho+\sqrt{-1}\beta\rho d\theta)\wedge(d\rho-\sqrt{-1}\beta\rho
d\theta)+\sum_{j>1}dz_j\wedge d\bar{z}_j
\end{eqnarray*}
Let $\epsilon=e^{\sqrt{-1}\beta\theta}(d\rho+\sqrt{-1}\beta\rho d\theta)$, we can write
\begin{equation}\label{compconic}
\omega=\sqrt{-1}\left(f
\epsilon\wedge\bar{\epsilon}+f_{\bar{j}}\epsilon\wedge
d\bar{z}_j+f_jdz_j\wedge\bar{\epsilon}+f_{i\bar{j}}dz_i\wedge
d\bar{z}_j\right)
\end{equation}
\begin{defn}
\begin{enumerate}
\item
A function $f$ is in $C^{,\gamma,\beta}(X,D)$ if $f$ is $C^\gamma$ on $X\setminus D$, and  locally near each point in $D$, $f$ is $C^{\gamma}$ in the coordinate $(\hat{\zeta}=\rho
e^{i\theta}=z_1|z_1|^{\beta-1}, z_j)$.

\item
A (1,0)-form $\alpha$ is in
$C^{,\gamma,\beta}(X,D)$ if $\alpha$ is $C^\gamma$ on $X\setminus D$ and  locally near each point in $D$, we have
$\alpha=f_1\epsilon+\sum_{j>1}f_j dz_j$ with $f_i\in C^{,\gamma,\beta}$ for $1\le i\le n$,
and $f_1\rightarrow 0$ as $z_1\rightarrow 0$.

\item A (1,1)-form $\omega$ is in $C^{,\gamma,\beta}(X, D)$ if $ \omega$ is $C^\gamma$ on $X\setminus D$ and near each point in $D$ we can write  $\omega$ as \eqref{compconic} $f, f_j, f_{\bar{j}}, f_{i\bar{j}}\in
C^{,\gamma,\beta}$, and $f_j, f_{\bar{j}}\rightarrow 0$ as
$z_1\rightarrow 0$.

\item A function $f$ is in $C^{2,\gamma,\beta}(X,D)$ if $f, \partial f,
\partial\bar{\partial}f$ are all in $C^{, \gamma,\beta}$.
\end{enumerate}
\end{defn}

It is easy to see that the above definitions do not depend on the particular choice of local complex chart. Donaldson set up the linear theory in \cite{Do11}.
\begin{prop}(\cite{Do11})\label{DoFred}
If $\gamma<\mu=\beta^{-1}-1$, then the inclusion
$C^{2,\gamma,\beta}(X, D)\rightarrow C^{,\gamma,\beta}(X, D)$ is compact. If
$\omega$ is a $C^{,\gamma,\beta}$ K\"{a}hler metric on $(X,D)$ then
the Laplacian operator for  $\omega$ defines a Fredholm map $\Delta_{\omega}:
C^{2,\gamma,\beta}(X, D)\rightarrow C^{,\gamma,\beta}(X, D)$.
\end{prop}
In order to consider the conical K\"{a}hler-Einstein metrics for
different cone angles at the same time, we define the following
space
\begin{defn}\label{admissible}
Fix a smooth metric $\omega_0$ in $c_1(X)$, we define the space of admissible functions to be
\[
\hat{\mathcal{C}}(X,D)=C^{2,\gamma}(X)\cup\bigcup_{0<\beta<1}\left(\bigcup_{0<\gamma<\beta^{-1}-1}C^{2,\gamma,\beta}(X,D)\right),
\]
and the space of admissible K\"ahler potentials to be
\[
\hat{\mH}(\omega_0)=\{\phi\in \hat{\mathcal{C}}(X,D)|\omega_\phi:=
\omega_0+\sddbar\phi>0\}.
\]
\end{defn}
It is clear that $\hat{\mH}(\omega_0)$ includes the space of smooth K\"ahler potentials
\[
\mH(\omega_0)=\{\phi\in C^{\infty}(X)|
\omega_0+\sddbar\phi>0\},
\]
and is contained in the bigger space of bounded $\omega_0$-plurisubharmonic functions $PSH_\infty(\omega_0)=PSH(\omega_0)\cap L^\infty$. By modulo constants the space of admissible K\"ahler metrics corresponding to $\mH(\omega_0)$ consists exactly $C^{\gamma, \beta}$ K\"ahler metrics on $(X, D)$.

We will need the following fundamental openness theorem proved by
Donaldson.
\begin{thm}[\cite{Do11}]\label{DoIFT}
Let $\beta_0\in(0,1)$, $\alpha<\mu_0=\beta_0^{-1}-1$ and suppose
there is a $C^{2, \alpha, \beta_0}$ conical K\"{a}hler-Einstein on $(X,(1-\beta_0)D)$.
If there is no nonzero holomorphic vector fields on $X$
tangent to $D$ then for $\beta$ suifficiently close to $\beta_0$
there is a $C^{2, \alpha, \beta}$ conical K\"{a}hler-Einstein metric on $(X,(1-\beta)D)$.
\end{thm}

\subsection{Energy functionals} \label{energy functionals}

In the analytic study of K\"ahler-Einstein metrics, people have defined various functionals. We need to extend them to $\hat{\mH}(\omega_0)$. More generally some functionals extend even to $PSH(\omega_0)\cap L^\infty(X)$.

\begin{defn}\label{defF0IJ}
For any $\phi\in PSH_\infty(\omega_0)$, we define the functionals
\begin{enumerate}
\item
\[
F_{\omega_0}^0(\phi)=-\frac{1}{(n+1)!}\sum_{i=0}^n\int_X\phi\omega_\phi^i\wedge\omega_0^{n-i}
\]
\item
\[
J_{\omega_0}(\phi)=F_{\omega_0}^0(\phi)+\int_X\phi\omega_0^n/n!
\]
\item
\[
I_{\omega_0}(\omega_\phi)=\int_X
\phi(\omega_0^n-\omega_{\phi}^n)/n!,
\]
\end{enumerate}
\end{defn}
By pluripotential theory the above are well-defined functionals. The
following facts are also well known.
\begin{prop}
\begin{enumerate}
\item If $\phi_t$ is a smooth path in $\mH(\omega)$, then
\begin{equation}\label{derF0}
\frac{d}{dt}F_\omega^0(\phi_t)=-\int_X\dot{\phi}\omega_\phi^n/n!,
\end{equation}
\item
\begin{equation}\label{IJrel}
\frac{n+1}{n}J_\omega(\phi)\le I_\omega(\phi)\le
(n+1)J_\omega(\phi),
\end{equation}
\end{enumerate}
\end{prop}
\begin{rem}\label{BCF0}
Equation \eqref{derF0} tells us that $F_\omega^0(\phi)$ is the
integral of Bott-Chern form. If we let $h$ be the Hermitian metric
on $K_X^{-1}$ such that $\omega_h:=-\sddbar\log h=\omega$. Denote
$h_\phi=h e^{-\phi}$. Connect $h$ and $h_\phi$ by any path $h_t=h
e^{-\phi_t}$. The corresponding path of curvature forms
$\omega_t=\omega_{h_t}=\omega+\sddbar\phi_t$ connects $\omega$ and
$\omega_\phi$. The Bott-Chern form is defined by
\[
BC\left(c_1(K_X^{-1})^{n+1}; h,h _{\phi}\right)=-\int_0^1 dt
(n+1)h_t^{-1}\dot{h}_t c_1(K_X^{-1},h_t)^{n}
=(n+1)\int_0^1dt\dot{\phi}\omega_t^{n}.
\]
So we have the following identify which we will use in Section \ref{pfspecial}:
\[
F_\omega^0(\phi)=-\frac{1}{(n+1)!}\int_X
BC\left(c_1(K_X^{-1})^{n+1}; h,h _{\phi}\right).
\]
\end{rem}
Recall there are two functionals whose Euler-Lagrange equation is the K\"{a}hler-Einstein equation: the Ding energy and the Mabuchi energy. We now extend their definition to $\hat{\mH}(\omega_0)$. For the smooth metric $\omega_0$ in $c_1(X)$, define
the twisted Ricci potential $H_{\omega_0,
(1-\beta)D}$ by
\[
Ric(\omega_0)-r(\beta)\omega_0-2\pi(1-\beta)\{D\}=\sddbar
H_{\omega_0,(1-\beta)D},\quad \int_X e^{H_{\omega_0,
(1-\beta)D}}\frac{\omega_0^n}{n!}=\int_X\frac{\omega_0^n}{n!}
\]
It is easy to see that up to a constant $H_{\omega_0,
(1-\beta)D}=h_{\omega_0}-(1-\beta)\log|s_D|^2$, where $h_{\omega_0}$
is the usual Ricci potential of $\omega_0$, defined by the following
equation:
\[
Ric(\omega_0)-\omega_0=\sddbar h_{\omega_0},\quad \int_X
e^{h_{\omega_0}}\frac{\omega_0^n}{n!}=\int_X\frac{\omega_0^n}{n!}.
\]
$|s_D|^2$ is the norm of the defining section of $D$ under the
Hermitian metric on $-K_X$ satisfying
$-\sddbar\log|s_D|^2=\omega_0$. We will use the following definition of volume in this paper:
\[
Vol(X)=\int_X\frac{\omega^{n}}{n!}=(2\pi)^n\frac{\langle c_1(K_X^{-1})^n, [X]\rangle}{n!},
\quad Vol(D)=\int_D\frac{\omega^{n-1}}{(n-1)!}=(2\pi)^{n-1}\frac{\langle c_1(K_X^{-1})^{n-1},[D]\rangle}{(n-1)!}.
\]
We first generalize the Mabuchi-energy and Ding-energy to the
conical setting. In the next section we will show that
log-Mabuchi-energy integrates log-Futaki invariant.
\begin{defn}\label{deflogfunc}
\begin{enumerate}
\item(log-Mabuchi-energy) For any $\phi\in \hat{\mH}_{\omega_0}$
\begin{eqnarray*}
\mM_{\omega_0, (1-\beta)D}(\omega_\phi)&=&
\int_X\log\frac{\omega_\phi^n}{e^{H_{\omega_0,(1-\beta)D}}\omega_0^n}\frac{\omega_\phi^n}{n!}+r(\beta)\left(\int_X\phi
\frac{\omega_\phi^n}{n!}+F_{\omega_0}^0(\phi)\right)+\int_X
H_{\omega_0,(1-\beta)D}\frac{\omega_0^n}{n!}\\
&=&\int_X\log\frac{\omega_\phi^n}{e^{H_{\omega_0,(1-\beta)D}}\omega_0^n}\frac{\omega_\phi^n}{n!}-r(\beta)(I-J)_{\omega_0}(\omega_\phi)+\int_X
H_{\omega_0,(1-\beta)D}\frac{\omega_0^n}{n!}.
\end{eqnarray*}
\item(log-Ding-energy)
\[
F_{\omega_0,(1-\beta)D}(\omega_\phi)=F_{\omega_0}^0(\phi)-\frac{Vol(X)}{r(\beta)}\log\left(\frac{1}{Vol(X)}\int_X
e^{H_{\omega_0,(1-\beta)D}-r(\beta)\phi}\frac{\omega_0^n}{n!}\right).
\]
We also call $r(\beta) F_{\omega_0,(1-\beta)}(\omega_\phi)$ the normalized log-Ding-energy.
\end{enumerate}

When $\beta=1$ these functionals go back to the original functionals on smooth manifolds, which we denote by $\mM_{\omega_0}$ and $F_{\omega_0}$ for simplicity.
\begin{proof}[Proof of Proposition \ref{interpolate}]
Rewrite the log-Mabuchi-energy as:
\begin{equation}\label{linearmabuchi}
\mM_{\omega_0,
(1-\beta)D}(\omega_\phi)=\int_X\log\frac{\omega_\phi^n}{\omega_0^n}\omega_\phi^n+r(\beta)\left(\int_X\phi\omega_\phi^n+F_{\omega_0}^0(\phi)\right)+\int_X
(h_{\omega_0}-(1-\beta)\log|s|^2)\frac{\omega_0^n-\omega_\phi^n}{n!}.
\end{equation}
We see immediately that the linearity of log-Mabuchi-energy follows
from the linearity of $r(\beta)=1-\lambda(1-\beta)$ in $\beta$ and
the relation $H_{\omega_0,(1-\beta_t)
D}=(1-t)H_{\omega_0,(1-\beta_0) D}+tH_{\omega_0,(1-\beta_1) D}+C_t$.
For the log-Ding-energy, let $\beta_t=(1-t)\beta_0+t\beta_1$. So by
H\"{o}lder inequality we get
\begin{eqnarray*}
\int_X e^{H_{\omega_0,(1-\beta_t)
D}-r(\beta_t)\phi}\frac{\omega_0^n}{n!}&=&
e^{C_t}\int_X \left(e^{H_{\omega_0,(1-\beta_0) D}-r(\beta_0)\phi}\right)^{1-t}\left(e^{H_{\omega_0,(1-\beta_1) D}-r(\beta_1)\phi}\right)^{t}\frac{\omega_0^n}{n!} \\
&\le& e^{C_t}\left(\int_X e^{H_{\omega_0,(1-\beta_0)
D}-r(\beta_0)\phi}\frac{\omega_0^n}{n!}\right)^{1-t}\left(\int_X
e^{H_{\omega_0,(1-\beta_1)
D}-r(\beta_1)\phi}\frac{\omega_0^n}{n!}\right)^{t}
\end{eqnarray*}
By taking logarithm and using the definition of the log-Ding-energy we get, we get
\[
r(\beta_t) F_{\omega_0,(1-\beta_t)D}\ge (1-t) r(\beta_0)
F_{\omega_0,(1-\beta_0)D}+t\cdot r(\beta_1) F_{\omega_0,
(1-\beta_1)D}.
\]
\end{proof}
\end{defn}

By studying the behavior of conical metrics near $D$, it is not hard to see that the above functionals $\mM_{\omega_1}(\omega_2)$, etc. are all well defined for any a $C^{\gamma_1, \beta_1}$ metric $\omega_1$ and $C^{\gamma_2, \beta_2}$ metric $\omega_2$.
\begin{prop} \label{energy functional property}
\begin{enumerate}
\item \label{EL}
The Euler-Lagrange equations of log-Mabuchi-energy and
log-Ding-energy are the same:
\[
(\omega_0+\sddbar\phi)^n=Ce^{-r(\beta)\phi}e^{H_{\omega_0,(1-\beta)D}}\omega_0^n
\]
\item\label{MFdif} The log-Mabuchi-energy and log-Ding-energy differ by a cycle.
\[
\mM_{\omega_0,(1-\beta)D}(\omega_\phi)=r(\beta)
F_{\omega_0,(1-\beta)D}(\omega_\phi)+\int_X
H_{\omega_0,(1-\beta)D}\frac{\omega_0^n}{n!}-\int_X
H_{\omega_\phi,(1-\beta)D}\frac{\omega_\phi^n}{n!}.
\]
\item\label{logMgeF} log-Mabuchi-energy is bounded from below by log-Ding-energy:
\begin{equation*}
\mM_{\omega_0,(1-\beta)D}(\omega_\phi)\ge r(\beta)
F_{\omega_0,(1-\beta)D}(\omega_\phi)+\int_X
H_{\omega_0,(1-\beta)D}\frac{\omega_0^n}{n!}
\end{equation*}
The equality holds if and only if $\omega_\phi$ is a conical
K\"{a}hler-Einstein metric on $(X,(1-\beta)D)$.
\item(co-cycle condition)\label{cocycle}
Assume $\omega_i$ are $C^{,\gamma_i,\beta_i}$ K\"{a}hler metrics on $(X, D)$, for $i=1,2,3$. Then
\[
\mM_{\omega_1,(1-\beta)D}(\omega_2)+\mM_{\omega_2,(1-\beta)D}(\omega_3)=\mM_{\omega_1,(1-\beta)D}(\omega_3)
\]
\[
F_{\omega_1,(1-\beta)D}(\omega_2)+F_{\omega_2,(1-\beta)D}(\omega_3)=F_{\omega_1,(1-\beta)D}(\omega_3)
\]
\end{enumerate}
\end{prop}
\begin{proof}
Items (\ref{EL}), (\ref{MFdif}) and (\ref{cocycle}) easily follows from the formula
relating twisted Ricci  potentials of two K\"{a}hler metrics.
\begin{eqnarray*}
H_{\omega_\phi,(1-\beta)D}&=&
H_{\omega_0,(1-\beta)D}+\log\frac{\omega_0^n}{\omega_\phi^n}-r(\beta)\phi-\log\left(\frac{1}{V}\int_Xe^{H_{\omega_0,(1-\beta)D}-r(\beta)\phi}\frac{\omega_0^n}{n!}\right)\\
&=&-\left(\log\frac{\omega_\phi^n}{e^{H_{\omega_0,(1-\beta)
D}-r(\beta)\phi}\omega_0^n}+\log\left(\frac{1}{V}\int_Xe^{H_{\omega_0,(1-\beta)
D}-r(\beta)\phi}\frac{\omega_0^n}{n!}\right)\right)
\end{eqnarray*}
Item (\ref{logMgeF}) follows from from concavity of logarithm.
\end{proof}

\begin{thm}[\cite{Bern3}]\label{kecrit}
If there exists a conical K\"{a}hler-Einstein metric
$\omega_{\beta}$ on $(X,(1-\beta)D)$, then $\omega_{\beta}$
obtains the minimum of log-Ding-energy
$F_{\omega_0,(1-\beta)D}(\omega_\phi)$.
\end{thm}
The idea is use the convexity of log-Ding-energy along a bounded
geodesic in $PSH_\infty(\omega_0)$(see \cite{Bern3}), and the fact
that $\omega_\beta$ is a critical point of log-Ding-energy. By
Proposition \ref{energy functional property}.\eqref{logMgeF} we get

\begin{cor}\label{mabuchimin}
$\omega_{\beta}$ also obtains the minimum of log-Mabuchi-energy
$\mM_{\omega_0,(1-\beta)D}$.
\end{cor}

\begin{rem} One technical point here is that it is more difficult to use convexity of log-Mabuchi-energy than that of log-Ding-energy, as it requires more regularity.

\end{rem}

The following properness of energy functions was introduced by Tian \cite{Ti97}.

\begin{defn}\label{defproper}
A functional $F:\mathcal{H}(\omega_0)\rightarrow \mathbb{R}$ is called
proper if there is an inequality of the type
\[
F(\omega_\phi)\ge f\left(I_{\omega_0}(\omega_\phi)\right), \mbox{ for
any } \phi\in \mH(\omega),
\]
where $f(t): \mathbb{R}_{+}\rightarrow\mathbb{R}$ is some monotone
increasing function satisfying
$\lim_{t\rightarrow+\infty}f(t)=+\infty$.
\end{defn}
Note that, by the inequalities \eqref{IJrel}, we could replace $I_{\omega_0}(\omega_\phi)$ by equivalent norms $J_{\omega_0}(\omega_\phi)$ or $(I-J)_{\omega_0}(\omega_\phi)$
in the above definition. Now we state a fundamental theorem by Tian which gives equivalent
criterion for the existence of K\"{a}hler-Einstein metric.
\begin{thm}[\cite{Ti97}]\label{Tianproper}

If $Aut(X,J)$ is discrete. There exists a K\"{a}hler-Einstein metric
on $X$ if and only if either $F_{\omega_0}(\omega_\phi)$ or
$\mM_{\omega_0}(\omega_\phi)$ is proper on $\mH(\omega_0)$.\end{thm}
The case when $Aut(X,J)$ is not discrete is more subtle. (We thank
Professor Gang Tian, Professor Jiang Song and Professor Robert J.
Berman for pointing out this point to us). The full general
statement is a conjecture by Tian \cite{Ti97}. But for our
application, we just need the following result obtained in
\cite{PSSW}. It gives a condition under which Tian's argument for
proving the properness works through.
\begin{thm}[\cite{PSSW}]\label{PSSW}
If $K\subset G$ is a closed subgroup whose centralizer in $G$
denoted by $\rm{Centr}_KG$ is finite, then $F_{\omega_0}$ is
linearly proper on $K$-invariant potentials.
\end{thm}

It is natural to extend the definition of properness to the conical case, where we simply replace $\mH(\omega_0)$ by $\hat{ \mH}(\omega_0)$.
By approximating admissible potentials by smooth potentials it is easy to see
\begin{lem}\label{testsm}
If log-Mabuchi-energy or log-Ding-energy is proper (resp. bounded from below) on the space of smooth K\"{a}hler potentials, then it's proper (resp. bounded
from below) on the space of admissible
K\"{a}hler potentials.
\end{lem}

\begin{lem}
Let $\omega_i$ be a $C^{,\gamma_i,\beta_i}$ metric. Then the norm defined by $J_{\omega_1}$ and
$J_{\omega_2}$ are equivalent, that is, there is a constant $C(\omega_1, \omega_2)$ such that for any other metric
$\omega_3\in \hat \mH(\omega)$,
\[
\left|J_{\omega_1}(\omega_3)-J_{\omega_2}(\omega_3)\right|\le
C(\omega_1,\omega_2)
\]
\end{lem}
\begin{proof}
Assume $\omega_2=\omega_1+\sddbar{\phi}$ and $\omega_3=\omega_2+\sddbar\psi$. Then
\begin{eqnarray}\label{difI-J}
J_{\omega_1}(\omega_3)-J_{\omega_2}(\omega_3)&=&
F_{\omega_1}^0(\phi+\psi)-F_{\omega_1+\sddbar\phi}^0(\psi)+\int_X(\phi+\psi)\omega_1^n/n!-\int_X\psi\omega_2^n/n!\nonumber\\
&=&F_{\omega_1}^0(\phi)+\int_X\phi\omega_1^n/n!+\int_X\psi(\omega_1^n-\omega_2^n)/n!\nonumber\\
&=&J_{\omega_1}(\omega_2)+{\textbf E}.
\end{eqnarray}
To estimate the term ${\textbf E}$ we do integration by part:
\begin{eqnarray*}
{\textbf E}&=&\frac{1}{n!}\int_X\psi
(\omega_1-\omega_2)\wedge\left(\sum_{i=0}^{n-1}\omega_1^{n-1-i}\wedge\omega_2^{i}\right)=\frac{1}{n!}\int_X
-\phi(\omega_3-\omega_2)\wedge
\left(\sum_{i=0}^{n-1}\omega_1^{n-1-i}\wedge\omega_2^{i}\right)
\end{eqnarray*}
\[
\left|{\textbf E}\right|\le \frac{1}{n!}\int_X\left|\phi\right|
(\omega_2+\omega_3)\wedge
\left(\sum_{i=0}^{n-1}\omega_1^{n-1-i}\wedge\omega_2^{i}\right)\le
2n\|\phi\|_{L^{\infty}}Vol(X).
\]
\end{proof}

By the cocycle relations and the above lemmas, we obtain

\begin{prop}
Assume $\omega_i$ is a $C^{\gamma_i, \beta_i}$ K\"ahler metric on $(X, D)$. Then $\mM_{\omega_1, (1-\beta)D}$(or $F_{\omega_1, (1-\beta)D}$) is proper if and only if
$\mM_{\omega_2, (1-\beta)D}$(or $F_{\omega_2, (1-\beta)D}$) is proper.
 \end{prop}

\subsection{Existence of conical K\"{a}hler-Einstein metric}\label{existence}
Here we review the result in \cite{JMRL} to solve the conical
K\"{a}hler-Einstein equation.
\begin{thm}(\cite{JMRL})\label{JMRL}
If the log-Mabuchi-energy is proper on $C^{2,\gamma,\beta}(X, D)$, then
there exists a conical K\"{a}hler-Einstein metric on $(X,(1-\beta)D)$.
\end{thm}
The idea is to use continuity method as in the proof in the smooth
case. Fix a  backgroubd conical K\"ahler metric on $(X, (1-\beta)D)$. So we consider a family of equations.
\begin{equation}\label{conicont}
(\omega+\sddbar\psi)^n=e^{H_{\omega,(1-\beta)D}-t\psi}\omega^n
\end{equation}
This is equivalent to the equation
\begin{equation}\label{contensor}
Ric(\omega_{\psi})=t\omega_\psi+(r(\beta)-t)\omega+(1-\beta)\{s=0\}.
\end{equation}
\begin{itemize}
\item Step 1: Start the continuity method.
Let \[ \mathcal{S}=\{t\in(-\infty,r(\beta)); \eqref{conicont} \mbox{
is solvable}\}.
\]
Then $\mathcal{S}$ is non-empty. This is achieved by solving the
equation \eqref{conicont} when $t\ll0$ using Newton-Moser iteration
method. See \cite{JMRL} for details.
\item Step 2: Openness of solution set. This follows from implicit
function theorem thanks to the Fredholm linear theory set up by
Donaldson in Theorem \ref{DoFred}.
\item Step 3: $C^0$-estimate. This is the same as in the smooth
case. (cf. \cite{BM87},\cite{Ti97}). We sketch the proof here. First
by taking derivative with respect to $t$ on both sides of
\eqref{conicont}, we get
\[
\Delta_t\dot{\psi}_t=-t\dot{\psi}-\psi
\]

where $\Delta_t=\Delta_{\omega_{\psi_t}}$. So
\[
\frac{d}{dt}(I-J)_\omega(\omega_\psi)=-\int_X
\psi\Delta\dot{\psi}\frac{\omega_\psi^n}{n!}=\int_X
(-\Delta_t-t)\dot{\psi}(-\Delta_t)\dot{\psi}\frac{\omega_\psi^n}{n!}\ge
0
\]
Now using \eqref{contensor}, one can calculate that
\begin{eqnarray*}
\frac{d}{dt}\mM_{\omega,(1-\beta)D}(\psi_t)&=&-\int_X
H_{\omega,(1-\beta)D}\Delta_{\omega_\psi}\dot{\psi}\frac{\omega_\psi^n}{n!}+\int_X r(\beta)\psi\frac{d}{dt}\frac{\omega_\psi^n}{n!}\\
&=&-\int_X n
(Ric(\omega)-r(\beta)\omega-(1-\beta)\{D\})\dot{\psi}\frac{\omega_\psi^n}{n!}-r(\beta)\frac{d}{dt}(I-J)_{\omega}(\omega_\psi)\\
&=&-\int_X n
t(\omega_\psi-\omega)\dot{\psi}\frac{\omega_\psi^n}{n!}-r(\beta)\frac{d}{dt}(I-J)_{\omega}(\omega_\psi)\\
&=&-(r(\beta)-t)\frac{d}{dt}(I-J)_\omega(\omega_\psi)\le 0
\end{eqnarray*}
So along the continuity path \eqref{contensor}, the
log-Mabuchi-energy is decreasing. By properness of log-K-energy, we
see that $(I-J)_\omega(\omega_\psi)$ is bounded from above.

Now the $C^0$-estimate follows from the following Proposition.
\begin{prop}(\cite{JMRL},\cite{Ti87})\label{c0byint}
\begin{enumerate}
\item
${\rm Osc}(\psi_t)\le \frac{1}{Vol(X)}I_{\omega}(\omega_{\psi_t})+C$ for some
constant $C$ independent of $t$.
\item (Harnack estimate)
$ -\inf_X \psi_t\le n\sup_X\psi_t$.
\end{enumerate}
\end{prop}
\item Step 3: $C^2$-estimate.
To get the $C^2$-estimate, we can use the Chern-Lu's inequality and
maximal principle. More precisely, let $\Xi=\log
tr_{\omega_\psi}\omega-\lambda \psi$. Then
\begin{equation}\label{chernlu}
\Delta_{\omega_\psi} \Xi\ge (C_1-\lambda n)+(\lambda-C_2
tr_{\omega_\psi}\omega).
\end{equation}
Here
\[
C_1=\inf_{p\in X, v\in T_pX}
\frac{Ric(\omega_\psi)(v,\bar{v})}{g(v,\bar{v})}, \quad
C_2=\sup_{p\in X, v,w\in
T_pX}\frac{R(\omega)(v,\bar{v},w,\bar{w})}{g(v,\bar{v})g(w,\bar{w})}
\]
where $Ric(\omega_\psi)$ (res. $R(\omega)$) is the Ricci curvature
(resp. curvature operator) of $\omega_\psi$ (resp. $\omega$). By the
equation \eqref{contensor}, $Ric(\omega_{\psi_t})\ge
t\omega_{\psi_t}$, so $C_1\ge t$. The other crucial point is the
bisectional curvature of $\omega$ is bounded from above. (cf.
Appendix of \cite{JMRL}).

To use the maximal principle in the conical setting, one can use
Jeffres's trick as in \cite{Je00}. We add the barrier function
$\epsilon|s|^{2\gamma'}$ for $0<\gamma'<\gamma$ so that
$\Xi+\epsilon|s|^{2\gamma'}$ obtains the maximum at $x_0'\not\in D$.
We then apply the maximal principle to $\Xi+\epsilon|s|^{2\gamma'}$
and let $\epsilon\rightarrow 0$.
\item Step 4: $C^{,2,\gamma,\beta}$-estimate. There is a Krylov-Evans' estimate
in the conical setting as developed in \cite{JMRL}. The proof is
similar to the smooth case but adapted to the conical(wedge)
H\"{o}lder space.
\item
Use the above a priori estimate, we prove the solution set $\mathcal{S}$ is closed. So $\mathcal{S}=(-\infty,1]$ and the equation
\eqref{conicont} is solvable.
\end{itemize}

\subsection{Alpha-invariant and small cone angles} \label{small angles}
In \cite{Berm} and \cite{JMRL}, Tian's alpha invariant \cite{Ti87}
was generalized to the conical setting. We will explain this
modification.
\begin{defn}[log alpha-invariant]
Fix a smooth volume form $\Omega$. For any K\"{a}hler class $[\omega]$, we define
\begin{eqnarray*}
&&\alpha([\omega], (1-\beta)D)=\\
&&\max\left\{\alpha>0; \exists 0<C_\alpha<+\infty \mbox{ s.t. }
\int_X
e^{-\alpha(\phi-\sup\phi)}\frac{\Omega}{|s_D|^{2(1-\beta)}}\le
C_\alpha \mbox{ for any } \phi\in PSH_{\infty}(X,[\omega])\right\}.
\end{eqnarray*}
\end{defn}
When $\beta=1$, we get Tian's alpha invariant $\alpha([\omega])$ in \cite{Ti87}. In the following, we will write $\alpha(L,(1-\beta)D)=\alpha(c_1(L),(1-\beta)D)$ for any line bundle $L$. For any $\alpha<\alpha(K_X^{-1},(1-\beta)D)$, using concavity of log
function, we can estimate, for any $\phi\in
\hat{\mH}(\omega_0)\subset PSH_{\infty}(\omega_0)$,
\begin{eqnarray*}
\log C_\alpha&\ge &\log\left(\frac{1}{V}\int_X
e^{-\alpha(\phi-\sup\phi)}\frac{e^{h_{\omega_0}}\omega_0^n}{n!|s_D|^{2(1-\beta)}}\right)
=\log\left(\frac{1}{V}\int_Xe^{-\alpha(\phi-\sup\phi)-\log\frac{|s_D|^{2(1-\beta)}\omega_\phi^n}{e^{h_{\omega_0}}\omega_0^n}}\frac{\omega_\phi^n}{n!}\right)\\
&\ge&-\frac{1}{V}\int_X
\log\left(\frac{|s_D|^{2(1-\beta)}\omega_\phi^n}{e^{h_{\omega_0}}\omega_0^n}\right)\frac{\omega_\phi^n}{n!}+\alpha\left(\sup\phi-\frac{1}{V}\int_X\phi\frac{\omega_\phi^n}{n!}\right)\\
&\ge&\frac{1}{V}\left(-\int_X\log\frac{\omega_\phi^n}{e^{H_{\omega_0,(1-\beta)D}}\omega_0^n}\frac{\omega_\phi^n}{n!}+\alpha
I_{\omega_0}(\omega_\phi)\right).
\end{eqnarray*}
In the last inequality, we used the expression for $H_{\omega_0,(1-\beta)D}=h_{\omega_0}-(1-\beta)\log|s_D|^2$. Now using the expression for $\mathcal{M}_{\omega,(1-\beta)D}$ in
Definition \ref{deflogfunc} and inequalities in \eqref{IJrel}, we
get
\begin{eqnarray*}
\mM_{\omega_0,(1-\beta)D}(\omega_\phi)&\ge& \alpha
I_{\omega_0}(\omega_\phi)-r(\beta)(I-J)_{\omega_0}(\omega_\phi)-C_\alpha'\\
&\ge&
\left(\alpha-r(\beta)\frac{n}{n+1}\right)I_{\omega_0}(\omega_\phi)-C_\alpha'
\end{eqnarray*}
So if
\begin{equation}\label{alphaproper}
\alpha(K_X^{-1},(1-\beta)D)>\frac{n}{n+1}r(\beta)=\frac{n}{n+1}\left(1-\lambda(1-\beta)\right),
\end{equation}
then log-Mabuchi-energy is proper for smooth reference metric. To
estimate the alpha-invariant, we use Berman's estimate:
\begin{prop}[\cite{Berm}] If we let $L_D$ denote the line bundle determined by the divisor $D$, we have the estimate for log-alpha-invariant:
\begin{eqnarray}\label{estalpha}
\alpha(K_X^{-1}, (1-\beta)D)&=&\lambda\alpha(L_D, (1-\beta)D)\ge
\lambda\min\{\beta, \alpha(L_D|_D), \alpha(L_D)\}\nonumber\\
&=&\min\{\lambda\beta, \lambda\alpha(L_D|_D),\alpha(K_X^{-1})\}>0.
\end{eqnarray}
\end{prop}
\begin{cor}\label{betaproper}
When $\lambda\ge 1$, if
\begin{equation}
0<\beta<\min\left(1,
(1-1/\lambda)+\frac{n+1}{n}\min\{\alpha(L_D|_D),
\lambda^{-1}\alpha(K_X^{-1})\}\right),
\end{equation}
then the log-Mabuchi-energy is proper.
In particular, when $0<\beta<1-\lambda^{-1}+\epsilon$ for $\epsilon=\epsilon(\lambda)\ll 1$, the log-Mabuchi-energy is
proper. When $\lambda<1$, we need to assume in addition that
$\beta>n(\lambda^{-1}-1)$.
\end{cor}
\begin{proof}
This follows from \eqref{alphaproper}, \eqref{estalpha} and the
relation
\[
\lambda\beta>\frac{n}{n+1}\left(1-\lambda(1-\beta)\right)\Longleftrightarrow
\beta> n(\lambda^{-1}-1)
\]
This is automatically true if $\lambda\ge 1$ and $\beta>0$.
\end{proof}
\begin{rem}
 If we use H\"{o}lder's inequality, we could get the estimate: $\alpha(K_X^{-1},(1-\beta)D))\ge \alpha(K_X^{-1})\beta>0$.  (Note that
 it's easy to get that $\lambda\ge \alpha(K_X^{-1})$ from the existence of smooth divisor $D\sim \lambda K_X^{-1}$) However, if we want to prove there always exists a conical
 K\"{a}hler-Einstein metric with small cone angle, then this estimate only works when $\lambda>1$ but not equal to 1.
 To see this, we study the inequality $ \alpha(K_X^{-1})\beta>\frac{n}{n+1}r(\beta)$. When $\lambda>1$, we get
 \[
\beta<(\lambda-1)/(\lambda-\frac{n+1}{n}\alpha(K_X^{-1}))=(1-\lambda^{-1})\left(1+\frac{\frac{n+1}{n}\lambda^{-1}\alpha(K_X^{-1})}{1-\frac{n+1}{n}\lambda^{-1}\alpha(K_X^{-1})}\right).
\]
So again when $\beta<1-\lambda^{-1}+\epsilon$ for $\epsilon=\epsilon(\lambda)\ll 1$, the log-Mabuchi-energy is proper. When $\lambda=1$, we get the condition
$\alpha(K_X^{-1})>\frac{n}{n+1}$. This condition is not always satisfied, and if it's true, then $X$ has a smooth K\"{a}hler-Einstein metric by \cite{Ti87}.
When $\lambda<1$ we don't get useful condition on $\beta\in (0,1)$.  On the other hand, Berman's estimate works when $\lambda\ge 1$.
\end{rem}

\begin{cor}[Berman,\cite{Berm}]\label{nohv}
When $\lambda\ge 1$, there is no holomorphic vector field on $X$
tangent to $D$.
\end{cor}
\begin{proof}
If $v$ is the holomorphic vector field tangent to $D$, then $v$
generate a one-parameter subgroup $\lambda(t)$. Log-Mabuchi-energy
is linear along $\sigma^*\omega$ with the slope given by the
log-Futaki-invariant. This is in contradiction to Corollary
\ref{betaproper}.
\end{proof}
\begin{rem}
This corollary was speculated by Donaldson in \cite{Do11}. This is
also proved using pure algebraic geometry in Song-Wang's recent work
\cite{SW}.
\end{rem}
\begin{proof}[Proof of Theorem \ref{main1}]
By the discussion above, when $\lambda\ge 1$, the $\mM_{\omega, (1-\beta)D}$ is proper for $\beta\in (0, 1-\lambda^{-1}+\epsilon)$ with some $\epsilon>0$.
On the other hand, when there is a conical K\"ahler-Einstein metric on $(X, (1-\beta_0)D)$, then $\mM_{\omega, (1-\beta_0)D}$ is bounded below.
So we can use Proposition \ref{interpolate}
to get the properness of log-Mabuchi-energy for any $\beta\in (0, \beta_0)$. Now
we use Theorem \ref{JMRL} to conclude. The openness follows from \cite{Do11}.
\end{proof}
\begin{proof}[Proof of Corollary \ref{conicproper}]
Assume there exists a conical K\"{a}hler-Einstein metric for
$0<\beta=\beta_0<1$. Since we assume $\lambda\ge 1$, there is no
holomorphic vector field on $X$ fixing $D$ by Corollary \ref{nohv}.
By Donaldson's implicit function theorem for conical
K\"{a}hler-Einstein metrics in \cite{Do11}, there exists a conical
K\"{a}hler-Einstein metric for $\beta=\beta_0+\epsilon$ when
$\epsilon\ll 1$. So the log-Mabuchi-energy is bounded for
$\beta=\beta_0+\epsilon$. Because log-Mabuchi-energy is proper for
$0<\beta\ll 1$. Then we can use interpolation result Proposition
\ref{interpolate} to conclude the log-Mabuchi-energy is proper for
$0<\beta<\beta_0+\epsilon$.
\end{proof}

\section{Obstruction to existence:
log-K-stability}\label{obstruction}

\subsection{Log-Futaki invariant and
log-K-(semi)stability}\label{seclogK}

Fix a smooth K\"{a}hler metric $\omega\in 2\pi c_1(X)$. Assume $D$
is a smooth divisor such that $D\sim_{\mathbb{Q}}-\lambda K_X$ for
some $\lambda>0\in \mathbb{Q}$. Assume $\mathbb{C}^*$ acts on $(X,
D)$ with generating holomorphic vector field $v$. There exists a
potential function $\theta_v\in C^{\infty}(X)$ satisfying
$\sqrt{-1}\bar{\partial}\theta_v=\iota_v\omega$. The log-Futaki
invariant was defined by Donaldson \cite{Do11}:
\begin{defn}[\cite{Do11}]
The log-Futaki invariant $F(X,(1-\beta)D)=F(X,(1-\beta)D, 2\pi
c_1(X))$ of the pair $(X,(1-\beta)D)$ in the class $2\pi c_1(X)$ is
function on the Lie algebra of holomorphic vector fields, such that,
for any holomorphic vector field $v$ as above, its value is
\begin{eqnarray}\label{definelog}
F(X, (1-\beta)D)(v)&=&F(2\pi c_1(X);
v)+(1-\beta)\left(\int_{2\pi D}\theta_v\frac{\omega^{n-1}}{(n-1)!}-\frac{Vol(2\pi D)}{Vol(X)}\int_X\theta_v\frac{\omega^n}{n!}\right)\nonumber\\
&=&-\int_X(S(\omega)-n)\theta_v\frac{\omega^n}{n!}+2\pi(1-\beta)\left(
\int_{ D}\theta_v\frac{\omega^{n-1}}{(n-1)!}-n\lambda
\int_X\theta_v\frac{\omega^n}{n!}\right)
\end{eqnarray}
\end{defn}
Log-Futaki invariant is an obstruction to the existence of conical
K\"{a}hler-Einstein metrics as explained in \cite{Li11}. When
$\lambda\ge 1$, one can show that there are no nontrivial
$\mathbb{C}^*$ action for the pair $(X, D)$. (See Corollary \ref{nohv}) To obtain the
obstruction for the existence we define the log-K-stability by
generalizing the original definition by Tian \cite{Ti97} and Donaldson \cite{Do02}.
\begin{defn}
A test configuration of $(X, L)$,  consists of
\begin{enumerate}
\item a normal scheme $\mathcal{X}$ with a $\mathbb{C}^*$-action;
\item a $\mathbb{C}^*$-equivariant line bundle $\mathcal{L}\rightarrow\mathcal{X}$
\item a flat $\mathbb{C}^*$-equivariant map $\pi: \mathcal{X}\rightarrow\mathcal{C}$, where $\mathbb{C}^*$ acts on $\mathbb{C}$ by
multiplication in the standard way;
\end{enumerate}
such that any fibre $X_t=\pi^{-1}(t)$ for $t\neq 0$ is isomorphic to
$X$ and $(X,L)$ is isomorphic to $(X_t,\mathcal{L}|_{X_t})$. The
test configuration is called normal if the total space $\mX$ is
normal.
\end{defn}
Any test configuration can be equivariantly embedded into
$\mathbb{P}^N\times\mathbb{C}^*$ where the $\mathbb{C}^*$ action on
$\mathbb{P}^N$ is given by a 1 parameter subgroup of
$SL(N+1,\mathbb{C})$. If $Y$ is any subvariety of $X$, the test
configuration of $(X, L)$ also induces a test configuration
$(\mathcal{Y}, \mathcal{L}|_{\mathcal{Y}})$ of $(Y, L|_Y)$ .

Let $d_k$, $\tilde{d}_k$ be the dimensions of $H^0(X, L^{k})$,
$H^0(Y, L|_Y^{\;k})$, and $w_k$, $\tilde{w}_k$ be the weights of
$\mathbb{C}^*$ action on $H^0(\mX_0, \mathcal{L}|_{\mX_0}^{\;k})$,
$H^0(\mathcal{Y}_0, \mathcal{L}|_{\mathcal{Y}_0}^{\;k})$, respectively. Then we have
expansions:
\[
d_k=a_0k^n+a_1k^{n-1}+O(k^{n-2}) ,\quad
w_k=b_0k^{n+1}+b_1k^n+O(k^{n-1})
\]
\[
\tilde{d}_k=\tilde{a}_0k^{n-1}+O(k^{n-2}), \quad
\tilde{w}_k=\tilde{b}_0k^n+O(k^{n-1})
\]
If the central fibre $\mX_0$ is smooth, we can use equivariant
differential forms to calculate the coefficients as in \cite{Do02}. Let
${\omega}$ be a smooth K\"{a}hler form in $2\pi c_1(L)$, and
$\theta_v=\mathcal{L}_v-\nabla_v$,  then
\begin{equation}\label{coef2}
(2\pi)^n a_0=\int_X\frac{{\omega^n}}{n!}=Vol(X);\; (2\pi)^n
a_1=\frac{1}{2}\int_X S({\omega})\frac{{\omega}^n}{n!}
\end{equation}
\begin{equation}\label{coef1}
(2\pi)^{n+1} b_0=-\int_X\theta_v\frac{{\omega}^n}{n!};\;
(2\pi)^{n+1} b_1=-\frac{1}{2}\int_X\theta_v
S({\omega})\frac{{\omega}^n}{n!}
\end{equation}
\begin{equation}\label{coef3}
(2\pi)^{n}
\tilde{a}_0=\int_{2\pi\mathcal{Y}_0}\frac{{\omega}^{n-1}}{(n-1)!}=Vol(2\pi\mathcal{Y}_0);\;(2\pi)^{n+1}\tilde{b}_0=-\int_{2\pi\mathcal{Y}_0}\theta_v\frac{{\omega}^{n-1}}{(n-1)!}
\end{equation}
Comparing \eqref{coef2}-\eqref{coef3} with \eqref{definelog}, we can
define the algebraic log-Futaki invariant of the given test
configuration to be
\begin{eqnarray}\label{logftk}
\frac{1}{(2\pi)^{n+1}}F((\mathcal{X}, \mathcal{Y}); \mathcal{L})&=&\frac{2(a_0b_1-a_1b_0)}{a_0}+(-\tilde{b}_0+\frac{\tilde{a}_0}{a_0}b_0)\nonumber\\
&=&\frac{a_0(2b_1-\tilde{b}_0)-b_0(2a_1-\tilde{a}_0)}{a_0}
\end{eqnarray}
\begin{defn}\label{dflgkstb}
$(X, Y, L)$ is log-K-semistable along $(\mathcal{X}, \mathcal{L})$
if $F(\mathcal{X}, \mathcal{Y}, \mathcal{L})\le0$. Otherwise, it's
unstable.

$(X,Y, L)$ is log-K-polystable along test configuration
$(\mathcal{X}, \mathcal{L})$ if $F(\mathcal{X}, \mathcal{Y},
\mathcal{L})<0$, or $F(\mX,\mY,\mL)=0$ and the normalization
$(\mathcal{X}^{\nu}, \mathcal{Y}^{\nu}, \mathcal{L}^{\nu})$ is a
product configuration.

$(X, Y, L)$ is log-K-semistable (resp. log-K-polystable) if, for any
integer $r>0$, $(X, Y, L^r)$ is log-K-semistable (log-K-polystable)
along any test configuration of $(X, Y, L^r)$.
\end{defn}
\begin{rem}
When $Y$ is empty, then the definition of log-K-stability becomes the
definition of K-stability. (\cite{Ti97}, \cite{Do02})
\end{rem}

\subsection{Log-Mabuchi-energy and log-Futaki-invariant}
\subsubsection{Integrate log-Futaki-invariant}
We now integrate the log-Futaki invariant to get log-Mabuchi-energy,
as already defined in the previous section. Fix a smooth K\"ahler
metric $\omega\in 2\pi c_1(X)$. Define the functional on
$\mH(\omega)$ as
\[
F_{\omega,D}^{0}(\phi)=-\int_0^1dt\int_{
D}\dot{\phi_t}\frac{\omega_{\phi_t}^{n-1}}{(n-1)!},
\]
where $\phi_t$ is a family of K\"ahler potentials connecting $0$ and $\phi$.
We can define the
log-Mabuchi-energy as
\begin{equation}\label{logMsm}
\mM_{\omega,(1-\beta)D}(\omega_\phi)=\mM_\omega(\omega_\phi)+2\pi(1-\beta)\left(-
F_{\omega,D}^{0}(\phi)+\frac{ Vol(
D)}{Vol(X)}F_\omega^0(\phi)\right),
\end{equation}
so that if $\omega_t=\omega+\sddbar\phi_t$ is a sequence of smooth
K\"{a}hler metrics in $2\pi c_1(X)$, then
\[
\frac{d}{dt}\mM_{\omega,(1-\beta)D}(\omega_\phi)=-\int_X(S(\omega_t)-n)\dot{\phi_t}\frac{\omega_\phi^n}{n!}+2\pi(1-\beta)\left(
\int_{
D}\dot{\phi}\frac{\omega_\phi^{n-1}}{(n-1)!}-n\lambda\int_X\dot{\phi}\frac{\omega_\phi^n}{n!}\right)
\]
\begin{prop}
The log-Mabuchi-energy can be written as
\begin{eqnarray*}
\mM_{\omega,(1-\beta)D}(\omega_\phi)&=&\int_X\log\frac{\omega_\phi^n}{\omega^n}\frac{\omega_\phi^n}{n!}-
r(\beta)(I_\omega-J_\omega)(\omega_\phi)+\int_X\left(
h_\omega-(1-\beta)\log|s|_h^2\right)\frac{\omega^n-\omega_\phi^n}{n!}\\
&=&\int_X\log\frac{\omega_\phi^n}{e^{H_{\omega,(1-\beta)D}}\omega^n}\frac{\omega_\phi^n}{n!}+r(\beta)\left(\int_X\phi\frac{\omega_\phi^n}{n!}+F_\omega^0(\phi)\right)+\int_X
H_{\omega,(1-\beta)D}\frac{\omega^n}{n!},
\end{eqnarray*}
so it agrees the definition in the previous section.
\end{prop}
\begin{proof}
Recall the Poincar\'{e}-Lelong equation $
\sddbar\log|s_D|_h^2=-\lambda\omega+2\pi\{D\} $. Then
\begin{eqnarray*}
F_{\omega,2\pi D}^0(\phi)&=&2\pi F_{\omega,D}^0(\phi)=-\int_0^1 dt
\int_{2\pi D} \dot{\phi_t} \omega_{\phi_t}^{n-1}/(n-1)!
\\&=&-\int_0^1 dt \int_X
\dot{\phi}_t(\sddbar\log|s_D|_h^2+\lambda\omega)\omega_{\phi_t}^{n-1}/(n-1)!
\\&=&-\int_0^1dt\int_X
\log|s_D|_h^2\frac{d}{dt}\frac{\omega_{\phi_t}^n}{n!}+\lambda\mathcal{J}_{\omega}^{\omega}(\omega_\phi)\\
&=&-\int_X
\log|s_D|_h^2\frac{\omega_\phi^n-\omega^n}{n!}+\lambda\mathcal{J}_{\omega}^{\omega}(\omega_\phi)
\end{eqnarray*}
Here, for any smooth closed (1,1)-form $\chi$, we define
\[
\mathcal{J}^{\chi}_{\omega}(\phi)=-\int_0^1dt\int_X
\dot{\phi}_t\chi\wedge\omega_{\phi_t}^{n-1}/(n-1)!
\]
By taking derivatives, it's easy to verify that
$nF_{\omega}^0(\phi)-\mathcal{J}_{\omega}^{\omega}=(I-J)_\omega(\omega_\phi)
$. So
\begin{eqnarray*}
&&\mM_{\omega,(1-\beta)D}(\omega_\phi)=\mM_\omega(\omega_\phi)+\frac{(1-\beta)Vol(2\pi D)}{Vol(X)}F_\omega^0(\phi)-(1-\beta) F_{\omega,2\pi D}^{0}(\phi)\\
&=&\mM_\omega(\omega_\phi)+(n\lambda)
(1-\beta)F_\omega^0(\phi)-(1-\beta)\lambda\mathcal{J}^{\omega}_{{\omega}}(\phi)+(1-\beta)\int_X\log|s_D|^2\frac{\omega_\phi^n-\omega^n}{n!}\\
&=&\mM_\omega(\omega_\phi)+\lambda(1-\beta)(I-J)_\omega(\omega_\phi)+(1-\beta)\int_X\log|s_D|^2\frac{\omega_\phi^n-\omega^n}{n!}.
\end{eqnarray*}
Then the statement follows from the expression for $\mM_\omega$ and
that  $H_{\omega_0, (1-\beta)D}=h_{\omega_0}-(1-\beta)\log|s_D|^2$.
\end{proof}

\subsubsection{Log-Futaki invariant and asymptotic slope of
log-Mabuchi-energy}

In this section, we generalize Sean Paul's work in \cite{Pa12} to
the conical setting and prove Theorem \ref{funstable} using the
argument from \cite{Ti97} and \cite{PT09}. Assume
$X\subset\mathbb{P}^N$ is embedded into the projective space and
$\omega_{FS}\in 2\pi c_1(\mathbb{P}^N)$ is the standard Fubini-Study
metric on $\mathbb{P}^N$. For any $\sigma\in SL(N+1,\mathbb{C})$,
denote $\omega_\sigma=\sigma^*\omega_{FS}|_X$. We first recall Sean
Paul's formula for Mabuchi-energy
$\mathcal{M}_{\omega}=\mathcal{M}_{\omega,0}$ on the space of
Bergman metrics.
\begin{thm}[\cite{Pa12}]\label{paul}
Let $Emb_k: X^n\hookrightarrow\mathbb{P}^N=\mathbb{P}^{N_k}$ be the
embedding by the complete linear system $|-k K_X|$ for $k$
sufficiently large
%. a smooth, linearly normal, complex algebraic variety of degree $\ge 2$.
Let $R_X^{(k)}$ denote the \textbf{X-resultant} (the Cayley-Chow
form of $X$). Let $\triangle_{X\times\mathbb{P}^{n-1}}^{(k)}$ denote
the \textbf{X-hyperdiscriminant} of format (n-1) (the defining
polynomial for the dual of $X\times\mathbb{P}^{n-1}$ in the Segre
embedding). Then there are continuous norms such that the
Mabuchi-energy restricted to the Bergman metrics is given as
follows:
\begin{equation}\label{paulform}
\frac{n!k^{n}}{(2\pi)^{n+1}}\cdot
\mM_\omega(\omega_\sigma/k)=\log\frac{\|\sigma\cdot\triangle_{X\times\mathbb{P}^{n-1}}^{(k)}\|^2}{\|\triangle_{X\times\mathbb{P}^{n-1}}\|^2}-\frac{\deg
(\triangle_{X\times\mathbb{P}^{n-1}}^{(k)})}{\deg
(R_X^{(k)})}\log\frac{\|\sigma\cdot R_X^{(k)}\|^2}{\|R_X^{(k)}\|^2}.
\end{equation}
\end{thm}
%The most important ingredient to get this is the following identity proved by
%Sean Paul:
%\begin{thm}[\cite{Pa12},Hyper-discriminant part in the K-energy]
%\begin{eqnarray*}
%(2\pi)^{n+1}\cdot
%\log\frac{\|\sigma\cdot\triangle_{X\times\mathbb{P}^{n-1}}^{(k)}\|^2}{\|\triangle_{X\times\mathbb{P}^{n-1}}^{(k)}\|^2}&=&(N+n-1)\int_{0}^{1}dt\int_{(X\times\mathbb{P}^{n-1})^{\vee}}
%\dot{\Phi}_{\sigma}\omega_{FS(\mathbb{P}^{\vee})}^{N+n-1}\\
%&=&\int_0^1dt\int_X
%\dot{\phi}_\sigma(n(n+1)\omega_\sigma^n-nRic(\omega_\sigma)\wedge\omega_\sigma^{n-1})
%\end{eqnarray*}
%\end{thm}
For some notes on Paul's proof, see \cite{Li10}. One ingredient in Sean Paul's formula is
\begin{lem}[\cite{Zh96},\cite{Pa04}]
There is a continuous norm on the Chow forms, which satisfies, for
any projective variety $X^n\subset\mathbb{P}^N$,
\begin{equation}\label{F0chow}
(2\pi)^{n+1}\cdot\log\frac{\|\sigma\cdot
R_X^{(k)}\|^2}{\|R_X^{(k)}\|^2}=(n+1)\int_0^1dt\int_X\dot{\phi}_\sigma\omega_\sigma^n=-(n+1)!k^{n+1}
F_\omega^0(\phi_\sigma/k).
\end{equation}
In particular, this holds when $X^n$ is replaced by $D^{n-1}$ and
$F_{\omega}^0$ is replaced by $F_{\omega,D}^0$.
\end{lem}
We also know the degree of Cayley-Chow forms:
\begin{lem}
%The degree of hyper-discriminant
%$\triangle_{X\times\mathbb{P}^{n-1}}$ and
The degree of Cayley-Chow form $R_X^{(k)}$ and $R_D^{(k)}$ are given by
%\begin{itemize}
%\item(degree of hyperdiscrimiant, \cite{Pa12})
%\begin{eqnarray*}
%\deg(\triangle_{X\times\mathbb{P}^{n-1}}^{(k)})&=&
%\deg((X\times\mathbb{P}^{n-1})^\vee)=\int_{X\times\mathbb{P}^{n-1}}c_{2n-1}(J(\MO_{X\times\mathbb{P}^{n-1}}(1,1)))\\
%&=&n(n+1-\underline{S}k^{-1})\deg(X,-k K_X)=n(n+1-nk^{-1})k^{n}n!
%Vol(X).
%\end{eqnarray*}
%where $J(\mathcal{O}_{X\times\mathbb{P}^{n-1}}(1,1))$ is the jet
%bundle of
%$\mathcal{O}_{X\times\mathbb{P}^{n-1}}(1,1)=\pi_1^*\mathcal{O}_{\mathbb{P}^N}(1)\otimes\pi_2^*\mathcal{O}_{\mathbb{P}^{n-1}}(1)$.
%\item(Degree of Cayley-Chow form)
\begin{eqnarray*}
\deg(R_X^{(k)})&=&(n+1)k^{n}\deg(X,K_X^{-1})=\frac{(n+1)!k^n}{(2\pi)^n}\cdot Vol(X),\\
\deg(R_D^{(k)})&=&nk^{n-1}\deg(D,K_X^{-1})=\frac{n!k^{n-1}}{(2\pi)^{n-1}}\cdot
Vol(D).
\end{eqnarray*}
%\end{itemize}
\end{lem}
Combining the formulas \eqref{logMsm}, \eqref{paulform}, and
\eqref{F0chow}, we get
\begin{cor}\label{logpaul}
\begin{eqnarray*}
\frac{n! k^{n}}{(2\pi)^{n+1}}\cdot
\mM_{\omega,(1-\beta)D}(\omega_\sigma/k)&=&\log\frac{\|\sigma\cdot\triangle_{X\times\mathbb{P}^{n-1}}^{(k)}\|^2}{\|\triangle_{X\times\mathbb{P}^{n-1}}^{(k)}\|^2}-\frac{\deg
(\triangle_{X\times\mathbb{P}^{n-1}}^{(k)})}{\deg
(R_X^{(k)})}\log\frac{\|\sigma\cdot
R_X^{(k)}\|^2}{\|R_X^{(k)}\|^2}\\
&&+(1-\beta)\left(\log\frac{\|\sigma\cdot
R_D^{(k)}\|^2}{\|R_D^{(k)}\|^2}-\frac{\deg(R_D^{(k)})}{\deg(R_X^{(k)})}\log\frac{\|\sigma\cdot
R_X^{(k)}\|^2}{\|R_X^{(k)}\|^2}\right)
\end{eqnarray*}
\end{cor}
For any one parameter subgroup $\lambda(t)=t^{A}\in
SL(N_k+1,\mathbb{C})$. Although the log-Mabuchi-energy is not convex
along $\lambda(t)$, the above Corollary says that it is the linear
combination of convex functionals. As a consequence, we have the
existence of asymptotic slope. Define
$\omega_{\lambda(t)}=\lambda(t)^*\omega_{FS}|_X$, and
$(\mX_0,\mD_0)=\lim_{t\rightarrow 0}\lambda(t)\cdot (X,D)$ in the
Hilbert scheme (which is the central fibre of the induced test
configuration introduced in section \ref{seclogK}). Then combine
Corollary \ref{logpaul} with the argument in \cite{PT09}, we have
the following expansion
\begin{prop}\label{logexpansion}
\begin{equation}\label{aslope}
\mM_{\omega,(1-\beta)D}(\omega_{\lambda(t)}/k)=\left(F+a\right)\log
t+O(1)
\end{equation}
where $F=F(X,(1-\beta)D; 2\pi c_1(X))(\lambda)$ is the log-Futaki
invariant. $a\ge 0\in \mathbb{Q}$ is nonnegative and is positive if
and only if the central fibre $\mX_0$ has generically non-reduced
fibre.
\end{prop}
\begin{rem}
In fact, if $\mX_0$ is irreducible, then by (\cite{Ti97},
\cite{PT09}) one can calculate that $a=c\cdot({\rm mult}(\mX_0)-1)$
for $c>0\in\mathbb{Q}$.
\end{rem}
Without loss of generality, we assume each homogeneous coordinate
$Z_i$ are the eigenvector of $\lambda(t)$ on $H^0(\mathbb{P}^N,\mathcal{O}(1))=\mathbb{C}^{N+1}$
with eigenvalues $\lambda_0=\dots=\lambda_K<\lambda_{K+1}\le\dots\le
\lambda_N$. Let $\omega_{\lambda(t)}=\omega_{FS}+\sddbar\phi_t$.
Then
\begin{equation}\label{dephi}
\phi_t=\log\frac{\sum_{i}t^{\lambda_i}|Z_i|^2}{\sum_i|Z_i|^2}
\end{equation}
There are three possibilities for $\mX_0$.
\begin{enumerate}
\item (non-degenerate case)
$\lim_{t\rightarrow 0}{\rm Osc}(\phi_t)\rightarrow +\infty$. By
\eqref{dephi}, this is equivalent to $\cap_{i=0}^{K}\{Z_i=0\}\bigcap
X\neq\emptyset$.
\item (degenerate case)
${\rm Osc}(\phi_t)\le C$ for $C$ independent of $t$. This is
equivalent to $\cap_{i=0}^{K}\{Z_i=0\}\bigcap X=\emptyset$. In this
case, $\mX_0$ is the image of $X$ under the projection
$\mathbb{P}^N\rightarrow\mathbb{P}^K$ given by $[Z_0,\dots,
Z_N]\mapsto [Z_0,\dots,Z_K,0,\dots,0]$ and there is a morphism from
$\Phi: X=\mX_{t\neq0}\rightarrow \mX_0$ which is the restriction of
the projection. There are two possibilities.
\begin{enumerate}
\item $\deg(\Phi)>1$. In this case, $\mX_0$ is generically
non-reduced. So $a>0$ in \eqref{aslope}.

Example: Assume $X^n\subset \mathbb{P}^N$ is in general position.
Then the generical linear subspace
$\mathbb{L}\cong\mathbb{P}^{N-n-1}$ satisfies $\mathbb{L}\cap
X=\emptyset$. Let $\mathbb{M}\cong\mathbb{P}^{n}$ be a complement of
$\mathbb{L}\subset\mathbb{P}^N$. Then the projection of $\Phi:
\mathbb{P}^N\backslash\mathbb{L}\rightarrow \mathbb{M}$ gives a
projection $\Phi: X\rightarrow\Phi(X)$ whose mapping degree equals the
algebraic degree of $X$.
\item $\deg(\Phi)=1$. In this case, $\mX_0$ is generically reduced and $a=0$.

Example: Assume $X^n\subset \mathbb{P}^N$ is in general position.
Assume $K\ge n+1$, then $N-K-1\le N-n-2$. So the generical linear
subspace $\mathbb{L}\cong\mathbb{P}^{N-K-1}$ satisfies
$\mathbb{L}\cap X=\emptyset$. Let $\mathbb{M}\cong\mathbb{P}^K$ be a
complement of $\mathbb{L}\subset\mathbb{P}^N$. Then the projection
of $\Phi: \mathbb{P}^N\backslash\mathbb{L}\rightarrow \mathbb{M}$
gives a projection $\Phi: X\rightarrow\Phi(X)$ with degree 1.
\end{enumerate}
\end{enumerate}
\begin{prop}\label{Mabuchistable}
As a functional on the space $\mH(\omega)$ of smooth K\"{a}hler
potentials, if $\mM_{\omega, (1-\beta)D}(\omega_\phi)$ is bounded from below
(resp. proper), then $(X,-K_X,(1-\beta)D)$ is log-K-semistable
(resp. log-K-stable).
\end{prop}
\begin{proof}
If log-Mabuchi-energy is bounded from below, then $F\le 0$ by the expansion \eqref{aslope} since $a\ge 0$.

Assume $\mM_{\omega,(1-\beta)D}(\omega_\phi)$ is proper on $\mH(\omega)$ in the
sense of Definition \ref{defproper}, then in particular it's proper
on the space of Bermgan potentials, so by  \cite{PT09}, in case 1 or 2(a), $F<0$.
In case 2(b), $(\mX,\mY,\mL)$ has vanishing log-Futaki invariant and its normalization is a product test configuration.
\end{proof}

\begin{rem} \label{Bermanstability}
By \cite{LX11}, we only need to test log-K-(poly)stability for normal test
configuration with central fibre being a klt pair
$(\mX_0,(1-\beta)\mD_0)$. Recently, Berman \cite{Berm2} used this fact to prove that a
K\"{a}hler-Einstein (log) $\mathbb{Q}$-Fano variety is (log) K-polystable. His
approach is based on the expansion of Ding-functional along any special test configuration. This is certainly
related to the general expansion stated in Corollary \ref{logexpansion}.
\end{rem}

\subsection{Log-slope stability and log-Fano manifold }
Recall that when $\lambda<1$, $r(\beta)=1-\lambda(1-\beta)$. So when
$\beta=0$, $r(\beta)=1-\lambda>0$. The metric in this case should
correspond to complete metric with infinite diameter and with
$Ric=1-\lambda>0$. This contradicts Myers theorem. So we expect when
$\beta$ is very small, there does not exist such conical
K\"{a}hler-Einstein metrics.

This is indeed the case. To see this, we first generalize
Ross-Thomas' slope stability to the log setting. (See \cite{RoTh},
\cite{Su11}). For any subscheme $Z\subset X$, we blow up the ideal
sheaf $\mathcal{I}_Z+(t)$ on $X\times\mathbb{C}$ to get the
degeneration of $X$ to the deformation to the normal cone $T_ZX$.
For the polarization, we denote $\mathcal{L}_c=\pi^*L-c E$, where
$E=P_Z(\mathbb{C}\oplus N_ZX)$ is the exceptional divisor, $0<c<$
Seshadri constant of $Z$ with respect to $X$. By Ross-Thomas \cite{RoTh}, we have
the identity:
\begin{eqnarray*}
H^0(\mX, \mL_c^{k})=H^0(X\times\mathbb{C}, L^k\otimes
((t)+\mathcal{I}_Z)^{ck})=\bigoplus_{j=0}^{ck-1} t^j H^0(X,
L^k\otimes \mathcal{I}_Z^{ck-j})\oplus t^{ck}\mathbb{C}[t]H^0(X,
L^k)
\end{eqnarray*}
So for $k$ sufficiently large,
\begin{eqnarray*}
H^0(\mX_0, \mL_c^{k})&=&H^0(X, L^k\otimes
\mathcal{I}_Z^{ck})\oplus\bigoplus_{i=0}^{ck}t^{j}H^0(X,L^k\otimes
\mathcal{I}_Z^{ck-j}/\mathcal{I}_Z^{ck-j+1})\\
&=&H^0(X,L^k\otimes\mathcal{I}_Z^{ck})\oplus\bigoplus_{i=0}^{ck}t^{j}\frac{H^0(X,
L^k\otimes\mathcal{I}_Z^{ck-j})}{H^0(X,
L^k\otimes\mathcal{I}_Z^{ck-j+1})}
\end{eqnarray*}
If
\[
\chi(X, L^k\otimes\mathcal{I}_Z^{xk})=a_0(x) k^{n}+a_1(x)
k^{n-1}+O(k^{n-1})
\]
then by the calculation in by Ross-Thomas in \cite{RoTh}, we know
that
\[
b_0=\int_0^c a_0(x)dx-ca_0,\quad b_1=\int_0^c
(a_1(x)+\frac{1}{2}a_0'(x))dx-ca_1
\]
Similarly,
\begin{eqnarray*}
H^0(\mY, \mL_c^k)&=&H^0(X\times\mathbb{C}, L^k\otimes
((t)+\mathcal{I}_Z)^{ck}\otimes\mathcal{O}_X/\mathcal{I}_Y)=H^0(Y\times
\mathbb{C}, L^k\otimes ((t)+\mathcal{I}_Z\cdot\mathcal{O}_Y)^{ck})
\end{eqnarray*}
And
\[
H^0(\mathcal{Y}_0,\mL_c^k)
=H^0(Y,L^k\otimes(\mathcal{I}_Z\cdot\mathcal{O}_Y)^{ck})\oplus\bigoplus_{i=0}^{ck}t^{j}\frac{H^0(Y,
L^k\otimes(\mathcal{I}_Z\cdot\mathcal{O}_Y)^{ck-j})}{H^0(X,
L^k\otimes(\mathcal{I}_Z\cdot\mathcal{O}_Y)^{ck-j+1})}
\]
So, by \cite{RoTh} again, if
\[
\chi(Y,
L^k\otimes(\mathcal{I}_Z\cdot\mathcal{O}_Y)^{xk})=\tilde{a}_0(x)
k^{n-1}+O(k^{n-2}),
\]
then
\[
\tilde{b}_0=\int_0^c \tilde{a}_0(x)dx-c \tilde{a}_0.
\]
So we can calculate the log-Futaki invariant as
\begin{eqnarray}\label{callogfut}
\frac{a_0}{(2\pi)^{n+1}} F(\mathcal{X}, \mathcal{Y}, \mathcal{L})&=&
2(a_0 b_1-a_1 b_0)+(\tilde{a}_0b_0-a_0\tilde{b}_0)\nonumber\\
&=&a_0(2b_1-\tilde{b}_0)-b_0(2a_1-\tilde{a}_0)\nonumber\\
&=& 2a_0 \left(\int_0^c
(a_1(x)-\frac{1}{2}\tilde{a}_0(x)+\frac{1}{2}a_0'(x))dx\right)-2(a_1-\tilde{a}_0/2)\int_0^c
a_0(x)dx.
\end{eqnarray}
In other words, we can define the log-slope invariant:
\begin{eqnarray*}
\mu^{\log}_c((X,Y); \mathcal{I}_Z)&=&\frac{\int_0^c
(a_1(x)-\frac{1}{2}\tilde{a}_0(x)+\frac{1}{2}a_0'(x))dx}{\int_0^c
a_0(x)dx}=\frac{\int_0^c
(a_1(x)-\frac{1}{2}\tilde{a}_0(x))dx+\frac{1}{2}(a_0(c)-a_0)}{\int_0^c
a_0(x)dx}\\
&=&\mu_c(X;\mathcal{I}_Z)-\frac{\int_0^c\tilde{a}_0(x)dx}{2\int_0^ca_0(x)dx}.
\end{eqnarray*}
\[
\mu^{\log}((X,Y))=\frac{a_1-\tilde{a}_0/2}{a_0}=-\frac{n}{2}\cdot\frac{(K_X+Y)\cdot
L^{n-1}}{L^n}=\mu_c(X)-\frac{nY\cdot L^{n-1}}{2 L^n}.
\]
\begin{defn}
We call $(X,Y)$ is log-slope-stable, if for any subscheme
$Z\subset X$, we have
\[
\mu^{\log}_c((X,Y); \mathcal{I}_Z)<\mu^{\log}((X,Y))
\]
\end{defn}
\begin{prop} \label{defnormalcone}Let $X$ be a  Fano manifold, and $D$ a Cartier divisor which is numerically equivalent to $-\lambda K_X$.
Then if $\lambda<1$, the pair $(X, (1-\beta)D)$ is not
log-slope-stable for $\beta<(\lambda^{-1}-1)/n$. As a consequence, in the log-Fano
case, the log-Mabuchi-energy is not bounded from below for very small angle.
\end{prop}
\begin{proof}
The idea is to look at the test configuration $\mX$ given by deformation
to the normal cone to $D$, as in \cite{Su11}. By Lemma \ref{Sesh}, the Seshadri constant of $(-K_X, D)$ is equal to
$c=1/\lambda$. We will calculate the Futaki invariant for the semi
test configuration polarized by $\mL=L(-\frac{1}{\lambda}D)$ with
$L=-K_X$ and show it is negative for $\beta<(\lambda^{-1}-1)/n$.

In our case, if we choose $Z=D\sim -\lambda K_X$, then the
calculation simplifies to
\[
a_0(x)=\frac{(L-xD)^n}{n!}=(1-x\lambda)^n\frac{(-K_X)^n}{n!}=(1-x\lambda)^n
a_0,\quad  a_0'(x)=-n\lambda(1-x\lambda)^{n-1} a_0
\]
\[
a_1(x)=\frac{-K_X\cdot
(L-xD)^{n-1}}{2(n-1)!}=(1-x\lambda)^{n-1}\frac{(-K_X)^n}{2(n-1)!}=\frac{n}{2}(1-x\lambda)^{n-1}a_0.
\]
We let $Y=(1-\beta)D$, then
\[
\tilde{a}_0(x) =\left\{ \begin{array}{lcr} 0&,& \mbox{ when } x>0\\
\frac{L^{n-1}\cdot (1-\beta) D}{(n-1)!}=n\lambda(1-\beta)
a_0&,&\mbox{ when } x=0
\end{array}\right.
\]
So
\[
\int_0^c a_0(x)dx=\frac{a_0}{\lambda(n+1)}(1-(1-c\lambda)^{n+1})
\]
\[
\int_0^c a_0'(x)dx=a_0(c)-a_0=a_0((1-c\lambda)^n-1).
\]
\[
\int_0^c a_1(x)dx=\frac{a_0}{2\lambda}(1-(1-c\lambda)^{n}), \quad
\int_0^c \tilde{a}_0(x)dx=0.
\]
So the log-Futaki invariant is equal to
\begin{eqnarray*}
\frac{a_0}{(2\pi)^{n+1}}
F&=&a_0^2(1-(1-c\lambda)^{n+1})\frac{n}{n+1}\left[(\lambda^{-1}-1)\left(\frac{n+1}{n}\cdot\frac{1-(1-c\lambda)^{n}}{1-(1-c\lambda)^{n+1}}-1\right)-\beta\right]\\
\end{eqnarray*}
So we get $ F\le 0\Longleftrightarrow \beta\ge \beta(\lambda, c)$,
where
\begin{eqnarray*}
\beta(\lambda, c)&=&(\lambda^{-1}-1)\left(\frac{n+1}{n}\frac{1-(1-c\lambda)^n}{1-(1-c\lambda)^{n+1}}-1\right)\\
&=&\frac{\lambda^{-1}-1}{n}\left(1-\frac{n+1}{\sum_{i=0}^n
(1-c\lambda)^{-i}}\right)
\end{eqnarray*}
For above formula for $\beta(\lambda,c)$ we easily get that
\[
\sup_{0<c<\lambda^{-1}}\beta(\lambda,
c)=\frac{\lambda^{-1}-1}{n}.
\]
So when $\beta<(\lambda^{-1}-1)/n$ there exists $c\in (0,\lambda^{-1})$ such that $(X,(1-\beta)D)$ is destablized by $cD$.
\end{proof}
\begin{exmp} \label{deg2slope}
On $\cP^2$,  when $D$ is a line, then $(X, (1-\beta)D)$ is unstable
for all $\beta\in [0,1)$; when $D$ is a conic, then $(X,
(1-\beta)D)$ is unstable for $\beta\in(0, 1/4)$, and it will be
proved below that it is semi-stable for $\beta=1/4$, and hence
poly-stable for $\beta\in(1/4, 1)$. On $\cP^1\times \cP^1$, when $D$
is a diagonal line, $(X, (1-\beta)D)$ is unstable for $\beta\in(0,
1/2)$. By viewing $\cP^1\times\cP^1$ as a double cover of $\cP^2$
along a conic curve (See Remark \ref{coverP2} in Section
\ref{proofconic} for details) we see these observations match. It is
an interesting question whether the bounds of $\beta$ given by the
above proposition is sharp for a smooth hypersurface $d$ in $\cP^n$
with $d<n+1$.
\end{exmp}
\begin{lem}\label{Sesh}
The Seshadri constant of $(-K_X, D)$ is equal to $\lambda^{-1}$.
\end{lem}
\begin{proof}
Note that $\mX_0=X\cup_{D_\infty}E$. Here $E\cong
P(N_D\oplus\mathbb{C})$ is the exceptional divisor and $D\cong
D_\infty\subset P(N_D\oplus\mathbb{C})$ is the divisor at infinity.
$\mL_c|_X=K_X^{-1}-cD=(1-c\lambda)K_X^{-1}$. This is ample if and
only if $c<\lambda^{-1}$.

On the otherhand,
$\mL_c|_{P(N_D\oplus\mathbb{C})}=\pi^*K_X^{-1}+c\mathcal{O}_E(1)=\pi^*K_X^{-1}+cD_\infty$.
Let $h$ be a Hermitian metric on $\mathcal{O}(D)$ such that
$\omega_h:=-\sddbar\log h$ is a K\"{a}hler form. Then if we define
$\Omega=\lambda^{-1}\pi^*\omega_h+c\sddbar\log(1+h)$, this gives a
smooth rotationally symmetric (1,1)-form on $E$. To write $\Omega$
in local coordinate, we choose two kinds of coordinate charts on $E$
which covers the neighborhood of zero section $D_0$ and infinity
section $D_\infty$ of $P(N_D\oplus\mathbb{C})$ respectively. To do
this, just choose local trivialization of $N_D|_D$ to get
holomorphic coordinate along the fibre, which is denoted by $\xi$.
Then $h=a|\xi|^2$ for some smooth positive definite function $a$.
Note that $\omega_h=-\sddbar\log a$. In this local coordinate one
can easily calculate that
\[
\Omega=(\lambda^{-1}-c\frac{a|\xi|^2}{1+a|\xi|^2})\omega_h+c\sqrt{-1}\frac{a}{(1+a|\xi|^2)^2}\nabla
\xi\wedge\overline{\nabla \xi}.
\]
where for simplicity we denote $\nabla\xi=d\xi+\xi a^{-1}\partial
a$. For the coordinate at infinity, we use coordinate change
$\eta=\xi^{-1}$, then
\[
\Omega=(\lambda^{-1}-c\frac{a}{a+|\eta|^2})\omega_h+c\sqrt{-1}\frac{a}{(|\eta|^2+a)^2}\nabla'\eta\wedge\overline{\nabla'\eta}
\]
with $\nabla'\eta=d\eta-\eta a^{-1}\partial a$. So we easily sees
that $\Omega$ is positive definite if an only if $c<\lambda^{-1}$.
The lemma clearly follows from the combination of above discussions.
\end{proof}

The following example is in the log-Calabi-Yau case ($\lambda=1$).
%We have more to say about this in the last section \ref{relSW}.
\begin{exmp}
Let $X=Bl_p\cP^2$, $D\in |-K_X|$ be a general smooth divisor. Choose
$Z=E$ to be the exceptional divisor. If we perform the operation of
deformation to the normal cone, the central fibre is given by
$\tilde{\mathcal{X}}_0=X\cup_{E=D_\infty}
\mathbb{P}(\mathbb{C}\oplus\mathcal{O}(-1))$. The Seshadri constant
equals $2$ and the line bundle $\mL_{2}$ contracts $X$ along its
fibration direction and the resulting test configuration has central
fibre $\mX_0=\mathbb{P}(\mathbb{C}\oplus\mathcal{O}(-1))\cong X$.
The boundary divisor on $\mX_0$ is given by $F+2D_\infty$ where $F$
is the fibre over the intersection point $F\cap E\in E=D_\infty$.
Denote $L=K_X^{-1}$ and $Y=(1-\beta)D$. Then the calculation
specializes to
\[
a_0(x)=\frac{(L-xZ)^2}{2}=4-x-\frac{x^2}{2},\quad
a_1(x)=\frac{-K_X\cdot (L-xZ)}{2}=4-\frac{x}{2}.
\]
\[
\tilde{a}_0(x)=\deg (L-xZ)|_{Y}=(1-\beta)(8-x).
\]
Using formula \eqref{callogfut}, it's easy to calculate the
log-Futaki invariant as
\[
F(\mX,\mY, \mL_2)=8\left(\frac{7\beta}{3}-2\right).
\]
This is negative if and only if $\beta<6/7$. This is compatible with
the calculation in \cite{Do11} (See also \cite{Li11}), where,
instead of taking deformation to normal cone, the same test
configuration is generated by one parameter subgroup in the torus
action.
\end{exmp}

\section{Special degeneration to K\"{a}hler-Einstein
varieties}\label{degKE}

\subsection{K\"{a}hler metrics on singular varieties}\label{singkahler}
We will first establish some standard notations following \cite{EGZ}.
\begin{defn}[$\mathbb{Q}$-Fano variety]
A normal variety $X$ is $\mathbb{Q}$-Fano if $X$ is klt and $-K_X$
is an ample $\mathbb{Q}$-Cartier divisor.
\end{defn}
Assume $X$ is an n-dimensional $\mathbb{Q}$-Fano variety. $D$ is a
smooth divisor such that $D\cap X^{sing}=\emptyset$. Define the
space of K\"{a}hler metrics on $\mathbb{Q}$-Fano varieties following
\cite{EGZ}. So a plurisubharmonic(psh) function $\phi$ on $X$ is an upper
semi-continuous function on $X$ with values in
$\mathbb{R}\cup\{-\infty\}$, which is not locally $-\infty$, and
extends to a psh function in some local embedding
$X\rightarrow\mathbb{C}^N$.  $\phi$ is said to be smooth(locally bounded) if there exists a smooth(locally bounded) local extension is smooth(bounded).
Similarly a smooth K\"ahler metric on $X$ is locally $\sddbar$ of a smooth plurisubharmonic function. We are only interested in the class of bounded
plurisubharmonic functions. Fix a smooth K\"ahler metric $\omega$ on  $X$,  we define
\[
PSH_{\infty}(\omega):=\{\phi\in L^{\infty}(X); \omega+\sddbar\phi\ge
0 \mbox{ and } \phi \mbox{ is u.s.c. }\}.
\]
\begin{rem}
Any function $\phi\in PSH_{\infty}(X,\omega)$ is of finite
self-energy in the sense of Definition 1.1 in \cite{EGZ}.
\end{rem}

\begin{rem}[Orbifold metric induces $L^{\infty}$ Hermitian metrics]

When the Cartier index of $K_X$ divides $r$, $K_X^{\otimes r}$ is a
line bundle. Any orbifold metric induces a Hermitian metric on
$K_X^{\otimes r}$ and hence on $K_X^{-\otimes r}$. In fact, for any
point $x$, we can choose local uniformization chart
$\tilde{U}\rightarrow U\ni x$ such that $U=\tilde{U}/G$ for some
finite group $G$ and we choose local coordinates $\{\tilde{z}_i\}$
on $\tilde{U}$. Define $r$=order of $G$. Then the Cartier index of
$K_X$ at $x$ divides $r$. The $r$-pluri-anticanonical form
$\tilde{\tau}=(\partial_{\tilde{z}_1}\wedge\dots\wedge
\partial_{\tilde{z}_n})^{\otimes r}$ on $\tilde{U}$ is $G$-invariant, so it
induces a local generator $\tau$ of $K_X^{-\otimes r}$ downstairs.
If we have an orbifold metric which is locally induced by a smooth
$G$-invariant metric $\tilde{g}$ on $\tilde{U}$. We just define the
Hermitian metric on $K_X^{-\otimes r}$ by requiring
$|\tau|^2=|\partial_{\tilde{z}}|_{\tilde{g}}^{2r}=\det(\tilde{g})^{r}$.
\begin{exmp}\label{1/4(1,1)}
Let $\mathbb{Z}_4$ acts on $\mathbb{C}^2$ by $\xi:
(\tilde{z}_1,\tilde{z}_2)\mapsto (\xi\tilde{z}_1,\xi\tilde{z}_2)$
where $\xi=\exp(2\pi\sqrt{-1}/4)$. Let
$X=\mathbb{C}^2/\mathbb{Z}_4$, then $X$ has an isolated singularity
of index 2, which is usually denoted by $\frac{1}{4}(1,1)$. We can
embed $X$ into $\mathbb{C}^5$ by defining
$u_i=\tilde{z}_1^{4-i}\tilde{z}_2^{i}$ for $i=0,\dots,4$.

We can choose the orbifold metric induced by the following smooth
metric on $\tilde{U}=\mathbb{C}^2$:
\[
\tilde{\omega}=\sddbar
(|z_1|^2+|z_1|^4+|z_2|^2)=(1+4|z_1|^2)dz_1\wedge
d\bar{z}_1+dz_2\wedge d\bar{z}_2
\]
Then $\tilde{\tau}=(\partial_{\tilde{z}_1}\wedge
\partial_{\tilde{z}_2})^{\otimes 2}$ induces a generator $\tau$ of
$K_X^{-\otimes 2}$ with
$|\tau|^2_{\tilde{g}}=(1+4|\tilde{z}_1|^2)^{2}=(1+4
|u_1|^{1/2})^{2}$.
\end{exmp}
By the above discussion, we see that the Hermitian metric determined
by an orbifold metric does not give rise to a smooth plurisubharmonic function. However, it is locally bounded, so there is no technical difficulty in dealing with them.
\end{rem}

\subsection{Degenerate Complex Monge-Amp\`{e}re equation on
K\"{a}hler manifolds with boundary} Let $M$ be a K\"{a}hler manifold
of dimension $n+1$ with smooth boundary $\partial M$. We will be
interested in solving the Dirichlet problem of complex
Monge-Amp\`{e}re equation on $M$. Now we have several results for
this problem. First we have the following existence of weak
solutions
\begin{thm}[\cite{Bern3}]\label{weakgeod}
Let $\omega$ be a nonnegative, smooth (1,1)-form on $X$. Assume $\phi_i\in \mathcal{PSH}(\omega)\cap C^{0}(X)$, $i=0,1$.
Then there exists a bounded geodesic $\Phi_t$ connecting $\phi_0$ and $\phi_1$. In other words,
there exists a bounded solution of the Dirichlet problem to the following homogeneous complex Monge-Amp\`{e}re equation on $X\times [0,1]\times S^1$.
\begin{equation}\label{geodeq}
\left\{
\begin{array}{l}
\pi^*\omega+\sddbar\Phi\ge 0,\\
(\pi^*\omega+\sddbar\Phi)^{n+1}=0, \\
\Phi|_{X\times\{0\}\times S^1}=\phi_0, \Phi|_{X\times\{1\}\times S^1}=\phi_1.
\end{array}
\right.
\end{equation}
\end{thm}
For the reader's convenience, we give the proof by Berndtsson.
\begin{proof}(\cite{Bern3})
Define
\[
\mathcal{K}=\{\Psi_t; \pi^*\omega+\sddbar\Psi\ge0, \lim_{t\rightarrow i}\Psi_t\le \phi_{i} \mbox{ for } i=0,1\}
\]
Define the Perron-Bremermann envolope as
\[
\Phi(x)=\sup\{\Psi(x); \Psi\in \mathcal{K}\}
\]
We want to prove $\Phi$ is a bounded solution of \eqref{geodeq}.

We first define $\tilde{\Psi}=max(\phi_0-A Re(t), \phi_1+A (Re(t)-1))$. For $A$ sufficiently big, $\tilde{\Psi}\in\mathcal{K}$.
So we can only consider functions in $\mathcal{K}$ which is greater than $\tilde{\Psi}$. Furthermore, since
there is an $S^1$-symmetry, we can only consider $S^1$ invariant functions in $\mathcal{K}$. For any such function $\Psi$, it is convex as function
of $Re(t)$, and it satisfies
\[
\phi_0-A Re(t)\le \Psi \le\phi_0+A Re(t), \quad \phi_1+A (Re(t)-1)\le\Psi\le\phi_1-A( Re(t)-1).
\]
So $\Psi$ has the right boundary values uniformly. The upper semicontinuous regularization $\Phi^*$ of $\Phi$ satisfy the same estimate and is plurisubharmonic. So
$\Phi^*\in\mathcal{K}$. So $\Phi=\Phi^*$ is $\omega$-plurisubharmonic. Since $\Phi^*$ is $\omega$-maximal, it satisfies the equation \eqref{geodeq}.
\end{proof}
\begin{rem}
The existence of $C^{1,1}$-geodesics connecting smooth K\"{a}hler
metrics was first proved by Chen in \cite{Chen1}. Since we want to
deal with mildly singular varieties, we choose to work with just
bounded solutions. There are many other important related works to
this result. See for example \cite{BT}, \cite{Ko98}.
\end{rem}
The following Theorem is now a well known fact which comes from many
people's work. (\cite{CKNS},\cite{Guan},\cite{Chen},\cite{Blo}). We
record a version appeared in \cite{Blo}.
\begin{thm}\label{blo}
Assume $\Omega_0>0$ is a K\"{a}hler form on $M$ and $F>0$ is a
smooth, strictly positive function. Consider the Dirichlet problem
of complex Monge-Amp\`{e}re equation
\begin{equation}\label{ndCMA}
(\Omega_0+\sddbar\Phi)^{n+1}=F \Omega_0^{n+1}, \quad \Phi|_{\partial
M}=\phi
\end{equation}
If there exists a smooth subsolution, that is, a $\Psi\in
C^{\infty}(M)$ such that $\Omega_0+\sddbar\Psi>0$ and
$(\Omega_0+\sddbar\Psi)^{n+1}\ge F\Omega_0^{n+1}$, then
\eqref{ndCMA} has a unique solution $\Phi\in C^{\infty}(M)$.
\end{thm}
We also record a result by Phong-Sturm.
\begin{thm}[\cite{PhSt}]\label{PhSt}
Assume $\Omega\ge 0$ and there exists a smooth divisor $E$ in the
interior of $M$ such that $\Omega >0$ on $M\backslash E$. Also
assume the line bundle $\mathcal{O}(E)$ has a Hermitian metric $H$,
such that $\Omega_{\epsilon}=\Omega+\epsilon\sddbar\log H>0$ for
$0<\epsilon\ll 1$ sufficiently small. Consider the following
homogeneous complex Monge-Amp\`{e}re equation
\begin{equation}\label{HCMA1}
(\Omega+\sddbar\Phi)^{n+1}=0,\quad \Phi|_{\partial M}=\phi.
\end{equation}
If there exists a subsolution $\Psi\in C^{\infty}(M)$ such that
$\Omega+\sddbar\Psi\ge 0$ and $\Psi|_{\partial M}=\phi$ , then
\eqref{HCMA1} has a bounded solution $\Phi\in L^{\infty}(M)$.
Moreover, $\Phi\in C^{1,\alpha}(M\backslash E)$ for any
$0<\alpha<1$.
\end{thm}
For the reader's convenience, we sketch the proof of Phong-Sturm's
theorem here. By subtracting the subsolution $\Psi$, we can assume
$\Psi=0$ on $M$ and $\phi=0$ on $\partial M$. Then we approximate
the degenerate complex Monge-Amp\`{e}re equation by a family of
non-degenerate equations. Let $\Omega_{s}=\Omega_0+s \sddbar\log H$.
Consider the family of equations:
\begin{equation}\label{CMAs}
(\Omega_s+\sddbar \Phi_s)^{n+1}=F_s \Omega_s^{n+1},\quad
\Phi_s|_{\partial M}=0.
\end{equation}
For now, we can choose any smooth function $F_s$ such that
$\|F_s\|_{C^0}\ll 1$. Since $\Psi=0$ is a subsolution of
\eqref{CMAs}, by the Theorem \ref{blo}, we can solve this equation
for $0<s\ll 1$. To get the solution of \eqref{HCMA1}, we want to
take limit of $\Phi_s$. So we need to establish uniform a priori
estimate for $\Phi_s$. First we want uniform $C^0$-estimate for
$\Phi_s$. To get the upper bound, note that
\[
\Delta_{\Omega_{\epsilon}}\Phi_s\ge
-tr_{\Omega_\epsilon}\Omega_s=-tr_{\Omega_\epsilon}(\Omega_\epsilon-(\epsilon-s)\sddbar\log
H)=-(n+1)+(\epsilon-s)tr_{\Omega^{\epsilon}}(\sddbar\log H)\ge -C_1,
\]
with $C_1$ independent of $s$. On the other hand we can solve the
Dirichlet boundary problem
\[
\Delta_{\Omega_\epsilon}\xi=-C_1, \quad \xi|_{\partial M}=0.
\]
Then by maximal principle, we get $ \Phi_s\le \xi$. On the other
hand, $\Psi=0$ is a subsolution of \eqref{CMAs}, by comparison
principle of complex Monge-Amp\`{e}re operator, we get that
$\Phi_s\ge 0$. So $\|\Psi\|_{C^0}\le C$. This uniform $C^0$-estimate
allows us to construct bounded solution to the homogeneous complex
Monge-Amp\`{e}re equation.

To get more regularity away from the divisor $E$, let
$\Phi_s^{\epsilon}=\Phi-(\epsilon-s)\log|\sigma|^2$, where $\sigma$
is the defining section of the line bundle $\mathcal{O}(E)$. We can
rewrite \eqref{CMAs} as
\[
(\Omega_{\epsilon}+\sddbar\Phi_s^{\epsilon})^{n+1}=\frac{F_s\Omega_s^{n+1}}{\Omega_\epsilon^{n+1}}\Omega_\epsilon^{n+1}=F_s^{\epsilon}\Omega_\epsilon^{n+1}
\]
Now we choose $F_s$ such that
$F_s^{\epsilon}:=F_s\Omega_s^{n+1}/\Omega_\epsilon^{n+1}$ is a
constant.

Phong-Sturm in \cite{PhSt} proved the following
partial-$C^1$ (adapting Block's estimate), partial $C^2$-estimate
(adapting Yau's estimate):
\[
|\nabla\Phi|\le C_1|\sigma|^{-\epsilon A_1},\quad |\Delta\Phi|\le
C_2|\sigma(z)|^{-\epsilon A_2}
\]
with $C_i$, $A_i$ independent of $s$. The partial
$C^{1,\alpha}$-regularity follows from these estimates.

\subsection{Proof of Theorem \ref{specialdeg}}\label{pfspecial}
Assume $\pi: (\mX,-K_{\mX/\mathbb{C}})\rightarrow \mathbb{C}$ is a
special degeneration. Assume for simplicity, $\mX$ has only finite
many isolated singularities $\{p_i\}$. Let
$\triangle=\{w\in\mathbb{C}; |w|\le 1\}$ be the unit disk and
$\mX_{\triangle}=\pi^{-1}(\triangle)$. We embed the special test
configuration equivariantly into $\mathbb{P}^N\times\mathbb{C}$:
\[
\phi_{\mX}: (\mX,
-K_{\mX/\mathbb{C}})\hookrightarrow\mathbb{C}\times (\mathbb{P}^N,
\mathcal{O}_{\mathbb{P}^N}(1)).
\]
We get a smooth $S^1$-invariant K\"{a}hler metric on
$\mX_{\triangle}$ by pulling back
$\Omega=\phi_{\mX}^*(\omega_{FS}+\sqrt{-1} dw\wedge d\bar{w})$. We
define the reference metric $X$ to be $\omega=\Omega|_{\mX_1}$,
where $\mX_1\cong X$ is the fibre above $\{w=1\}$.  For any $\phi\in
C^{\infty}(X)$, such that $\omega+\sddbar\phi>0$, we are going to
solve the homogeneous Monge-Amp\`{e}re equation
\begin{equation}\label{sHCMA}
(\Omega+\sddbar\Phi)^{n+1}=0, \Phi|_{S^1\times X}=\phi.
\end{equation}
\begin{prop}
There exists bounded solution $\Phi$ for \eqref{sHCMA}. $\Phi\in
C^{1,\alpha}(\mX\backslash\{p_i\})$.
\end{prop}
\begin{proof}
If $\mX$ is smooth, then this follows from B\l{}ocki's Theorem
(Theorem \ref{blo}). If $\mX$ is singular, we can choose a
equivariant resolution $\pi: \tilde{\mX}\rightarrow\mX$. Then we
solve the equation on $\tilde{\mX}$:
\begin{equation}\label{rHCMA}
(\tilde{\Omega}+\sddbar\tilde{\Phi})^{n+1}=0,
\tilde{\Phi}|_{S^1\times X}=\phi
\end{equation}
with $\tilde{\Omega}=\pi^*\Omega$ being a smooth, closed,
non-negative form. By the following Proposition, we have smooth
subsolution for \eqref{rHCMA}. So by Phong-Sturm's result (Theorem
\ref{PhSt}), we can get bounded solution $\tilde{\Phi}$ of
\eqref{rHCMA} and, moreover, $\tilde{\Phi}$ is $C^{1,\alpha}$ on
$\tilde{X}\backslash E$, where $E$ is exceptional divisor. Because
$\tilde{\Phi}$ is plurisubharmonic along the fibres of the
resolution which are compact subvarieties, so $\tilde{\Phi}$ is
constant on the fibre of the resolution and hence $\tilde{\Phi}$
descends to a solution $\Phi$ of \eqref{sHCMA}.
\end{proof}
As pointed out in the above proof, to apply Theorem \ref{PhSt}, we
need to know the existence of subsolutions. Let
$\mX^*=\mX\backslash\mX_0$. To construct such subsolution, we first
note that there is an equivariant isomorphism
\begin{equation}\label{eqisom}
\rho: \mathbb{C}^*\times X\cong\mX^*\hookrightarrow \mX.
\end{equation}
\begin{prop}
For any smooth K\"{a}hler potential $\phi$, there exists a smooth
$S^1$-invariant smooth K\"{a}hler metric $\Omega_{\Psi}$ on
$\mX_{\triangle}$ such that $\rho^*\Omega_\Psi|_{S^1\times
X}=\pi_2^*\omega_\phi$. As a consequence, $\Psi$ is a subsolution of
the homogeneous Monge-Amp\`{e}re equation \eqref{sHCMA}.
\end{prop}
\begin{rem}
 Similar result was proved in \cite{PhSt}. For the reader's convenience, we give a proof here.
\end{rem}
\begin{proof}
Under the isomorphism \eqref{eqisom}, we can write
\[
\pi_2^*\omega_\phi+\sqrt{-1}dw\wedge d\bar{w}=\Omega+\sddbar\Psi_0
\]
by taking $\Psi_0=-\log(h_\phi/\rho^*\phi_{\mX}^* h_{FS})$. Note
that this only holds on $\mathbb{C}^*\times X$. Now let $ \eta(w)$
be a smooth cut-off function such that $ \eta(w)=1$ for $|w|\le 1/3$
and $ \eta(w)=0$ for $|w|\ge 2/3$. Now we define a new metric on
$\mathbb{C}^*\times X$:
\begin{eqnarray*}
\Omega+\sddbar\Psi:&=&\pi_2^*\omega_\phi+\sqrt{-1}dw\wedge
d\bar{w}-\sddbar( \eta(w)\Psi_0)+a\sqrt{-1}dw\wedge d\bar{w}\\
&=&\Omega+\sddbar (\Psi_0-\eta(w)\Psi_0+a|w|^2).
\end{eqnarray*}
In other words we let $\Psi=(1-\eta(w))\Psi_0+a|w|^2+c$ for some
constant $c$.

We will show when $\mathbb{R}\ni a\gg 1$ is chosen to be big enough,
then we get a smooth K\"{a}hler metric on $\mX_{\triangle}$ with the
required condition.

For $|w|\ge 2/3$, $\Omega_\Psi=\pi^{*}\omega_\phi+a\sqrt{-1}dw\wedge
d\bar{w}$. When $|w|\le 1/3$, $\Omega_\Psi=\Omega+a\sqrt{-1}dw\wedge
d\bar{w}$. we can use the glue map $\rho$ to get a smooth
$S^1$-invariant K\"{a}hler metric on $\pi^{-1}(\{|w|\le 1/3\})$. So
$\Omega_\Psi$ is a smooth $S^1$-invariant K\"{a}hler metric for
$|w|\le 1/3$ or $|w|\ge 2/3$. We now need to consider the behavior
of $\Omega_\Psi$ at any point $p\in \mathbb{C}^*\times X$ such that
$1/3<|w(p)|<2/3$.
\begin{eqnarray*}
\Omega_\Psi&=&\pi^*\omega_\phi- \eta\sddbar\Psi_0-\Psi_0 \sddbar
\eta-\sqrt{-1}\left(\partial
\eta\wedge\bar{\partial}\Psi_0+\partial\Psi_0\wedge\bar{\partial}
\eta\right)+(a+1)\sqrt{-1}dw\wedge
d\bar{w}\\
&\ge& (1- \eta)(\omega_\phi+\sqrt{-1}dw\wedge d\bar{w})+
\eta\Omega-\epsilon\sqrt{-1}\partial\Psi_0\wedge\bar{\partial}\Psi_0+(a-\epsilon^{-1}|\eta_{w}|^2-\Psi_0\eta_{w\bar{w}})\sqrt{-1}
dw\wedge d\bar{w}
\end{eqnarray*}
Note that the first two terms together are strictly positive
definite. Because on $\mX_{|w|\ge 1/3}=\pi^{-1}(\{|w|\ge 1/3\})$,
$\Psi_0$ is a well defined smooth function there. So we can choose
$\epsilon$ sufficiently small and $a$ sufficiently big such that
this is a positive form on $\mX_{|w|\ge 1/3}$.
\end{proof}
\begin{proof}[Proof of Theorem \ref{specialdeg}]
There exists a metric $h_{\Omega}$ on $K_{\mX/\mathbb{C}}^{-1}$ such
that $\Omega=-\sddbar\log h_{\Omega}$. $h_{\Omega}$ defines a volume
form on each fibre. If we choose local coordinate $\{z_i\}$ on
$\mX_t$ and denote $\partial_{z}=\partial_{z_1}\wedge
\dots\wedge\partial z_n$ and $dz=dz_1\wedge\dots\wedge dz_n$. Then
the volume form is given by
\[
dV(h_{\Omega}|_{\mX_t})=\left|\partial_{z}\wedge
\overline{\partial_z}\right|_{h_{\Omega}|_{\mX_t}}^{2}dz\wedge
\overline{dz}
\]
Let $\mathcal{S}$ be the defining section of the divisor
$\mathcal{D}$. Fix the Hermitian metric $|\cdot|$ on
$\mathcal{O}_{\mX}(\mD)$ such that
$-\sddbar\log|\cdot|^2=\lambda\Omega$.

Let $\omega_t=\Omega|_{\mX_t}$. To prove the lower boundedness of log-Ding-functional $F_{\omega,(1-\beta)D}$, by Lemma \ref{testsm}, we only need to consider smooth K\"{a}hler
potentials. For any smooth potential $\phi\in C^{\infty}(X)$, we solve the homogeneous complex Monge-Amp\`{e}re equation \eqref{sHCMA} to get the geodesic ray $\Phi$.
Then consider the function on the base defined by
\[
f(t)=F_{\omega_t}^0(\Phi|_{\mX_t})-\frac{V}{r(\beta)}\log\left(\frac{1}{V}\int_{\mX_t}e^{-r(\beta)\Phi}\frac{dV(h_{\Omega}|_{\mX_t})}{|\mathcal{S}|^{2(1-\beta)}}\right)
\]
\begin{claim}
$f(t)$ satisfies $\Delta f\ge 0$.
\end{claim}
Assuming the claim, let's finish the proof of Theorem \ref{specialdeg}. By maximal principle of subharmonic function, we have
\[
F_{\omega_1,(1-\beta)D}^{X}(\phi)=f(1)=\max_{t\in \partial \triangle}f(t)\ge f(0)=F_{\omega_0,(1-\beta)\mathcal{D}_0}^{\mX_0}(\Phi|_{\mX_0})
\]
Now since by assumption, there exists a conical K\"{a}hler-Einstein metric $\widehat{\omega}_{KE}=\omega_0+\sddbar\widehat{\phi}_{KE}$ on $(\mX_0,(1-\beta)\mD_0)$. By Berndtsson's
Theorem \ref{kecrit}, we have
\[
F_{\omega_0,(1-\beta)\mD_0}(\Phi|_{\mX_0})\ge F_{\omega_0,(1-\beta)\mD_0}(\widehat{\phi}_{KE})
\]
So combining the above two inequality, we indeed get the lower bound of log-Ding-energy:
\[
F_{\omega_1,(1-\beta)D}^{X}(\phi)\ge F_{\omega_0,(1-\beta)\mD_0}^{\mX_0}(\widehat{\phi}_{KE})
\]
Now, to prove the claim, we write $f(t)$ as parts: $f(t)={\rm I}+{\rm II}$:
\begin{eqnarray*}
{\rm I}&=&F_{\omega_t}^0(\Phi|_{\mX_t})=-\frac{1}{(n+1)!}\int_X
BC(\Omega^{n+1}, (\Omega+\sddbar{\Phi})^{n+1}),\\
{\rm
II}&=&-\frac{V}{r(\beta)}\log\left(\frac{1}{V}\int_{\mX_t}e^{-r(\beta)\Phi}\frac{dV(h_{\Omega}|_{\mX_t})}{|\mathcal{S}|^{2(1-\beta)}}\right)
\end{eqnarray*}
For part ${\rm I}$, see Remark \ref{BCF0}. We use the property of
Bott-Chern form and the geodesic equation to get that,
\begin{eqnarray*}
\sddbar {\rm I}&=&-\frac{1}{(n+1)!}\int_{\mX_t} \sddbar
BC(\Omega^{n+1},
(\Omega+\sddbar{\Phi})^{n+1})\\
&=&-\frac{1}{(n+1)!}\int_{\mX_t} (\Omega+\sddbar\Phi)^{n+1}-\Omega^{n+1}\\
&=&\frac{1}{(n+1)!}\int_{\mX_t}\Omega^{n+1}\ge 0
\end{eqnarray*}
For part ${\rm II}$, we can write locally $1=\partial_z\otimes dz$ in the decomposition
$\mathcal{O}_{\mX}=-K_{\mX/\mathbb{C}}+K_{\mX/\mathbb{C}}$. Then we
think $1\in\mathcal{O}_{\mathbb{C}}$ is a holomorphic section in
$\pi_*\mathcal{O}_\mX=\mathcal{O}_{\mathbb{C}}$.
\[
{\rm II}=-\frac{V}{r(\beta)}\log \|1\|_{L^2}^2
\]
where $\|\cdot\|_{L^2}^2$ is the $L^2$-metric induced by the
singular metric
$H=h_{\Omega}e^{-r(\beta)\Phi}/|\mathcal{S}|^{2(1-\beta)}$ on
$-K_{\mX/\mathbb{C}}$.  Then the subharmonicity is given by the next
proposition.
\end{proof}
\begin{prop}
${\rm II}$ is a subharmonic function of t.
\end{prop}
\begin{proof}
First note that
\begin{eqnarray*}
-\sddbar\log
H&=&\Omega+(1-\lambda(1-\beta))\sddbar\Phi+(1-\beta)(-\lambda\Omega+\{\mathcal{D}=0\})\\
&=&(1-\lambda(1-\beta))\Omega_{\Phi}+(1-\beta)\lambda\Omega+(1-\beta)\{\mathcal{D}\}
\end{eqnarray*}
is a positive current. If $\mX$ is smooth, then the subharmonicity
follows immediately from Berndtsson's result in \cite{Bern2}. In our
case, $\mX$ has isolated singularities. We can use divisors to cut
out this singularity and reduce the problem to smooth fibrations of
Stein manifolds. We apply Berndtsson-Paun's argument in \cite{BePa}.
They construct a sequence of smooth fibrations $\pi_j:
\mX_j\rightarrow \mathbb{C}$, such that
\begin{enumerate}
\item
$\pi_j$ is a smooth fibration. Each fibre is a Stein manifold.
\item
As $j\rightarrow+\infty$, $\{\mX_j\}$ form an exhaustion of $\mX$.
\end{enumerate}
%By result in \cite{Bern1}, $-\log\|1\|_k^2$ is plurisubharmonic, where $\|\cdot\|_k$ is the $L^2$-norm on the
%direct image $(\pi_k)_*(\mathcal{O}_{\mX_k})$. Indeed,
Note that in our equivariant setting, we can also require the
$\mX_j$ is $\mathbb{C}^*$-invariant.

In \cite{BePa} Berndtsson-Paun proved that the relative Bergman
kernel metric $h_j$ of the bundle
$\mathcal{O}_{\mX_j}=K_{\mX_j/\mathbb{C}}+(-K_{\mX_j/\mathbb{C}})$
has semipositive curvature current. In other words, $-\log
|1|_{h_j}^2$ is plurisubharmonic on $\mX_k$. If we use
$\|\cdot\|_{j}$ to denote the $L^2$-metric on
$(\pi_j)_*\mathcal{O}_{\mX_j}$ induced by $H=h_{\Omega}
e^{-r(\beta)\Phi}/|\mathcal{S}|^{2(1-\beta)}$ on
$K_{\mX_j/\mathbb{C}}^{-1}$ and $\mathcal{K}_j(z,z)$ to denote the
relative Bergman kernel of
$\mathcal{O}_{\mX_j}=K_{\mX_j/\mathbb{C}}+(-K_{\mX_j/\mathbb{C}})$ ,
then
\[
\mathcal{K}_{j}(z,z)=\max\{|f|^2; \|f\|_{j}\le 1\},\quad
|1|_{h_j}^2=\frac{1}{\mathcal{K}_{j}(z,z)}.
\]
Now, as showed by Berndtsson-Paun, the Bergman kernel $\mathcal{K}$
of $\mathcal{O}_{\mX}=K_{\mX/\mathbb{C}}+(-K_{\mX/\mathbb{C}})$ is
the decreasing limit of the Bergman kernel $\mathcal{K}_j$, and
hence the Bergman kernel metric on $\mathcal{O}_{\mX}$ also has
semipositive curvature current. Since, for any $t\in \mathbb{C}$,
$H^0(\mX_t,\mathcal{O}_{\mX}|_{\mX_t})=\mathbb{C}$, using the
extremal characterization of the relative Bergman kernel, it's
straight to verify that the relative Bergman kernel metric (BK) on
$\mathcal{O}_\mX=K_{\mX/\mathbb{C}}+(-K_{\mX/\mathbb{C}})$ is given
by $|1|_{BK}^2=\frac{1}{\mathcal{K}(z,z)}=\|1\|_{L^2}^2$ which is a
pull-back function from the base $\mathbb{C}$. So we get that ${\rm
II}=-\log\|1\|_{L^2}^2$ is plurisubharmonic on the disk $\{|w|\le
1\}$. Note that, since ${\rm II}$ is rotationally symmetric, this is
same as saying that ${\rm II}$ is a convex function of $t=|w|$.
\end{proof}

\begin{rem}
When the central fiber is smooth, this is a theorem of Chen
\cite{Chen}, where a more general statement concerning constant
scalar curvature K\"ahler metrics is proved,  using the weak
convexity of  Mabuchi functional on the space of K\"ahler metrics.
It seems difficult to adapt Chen's argument to the singular setting.
The advantage here(in the log setting) is to use Ding's functional,
which requires much weaker regularity of the geodesics. A
fundamental result of Berndtsson \cite{Bern3} says that the Ding
functional is genuinely geodesically convex. This technique has been
demonstrated in \cite{Bern3},  \cite{BBEGZ} and \cite{Berm2}.
\end{rem}
\begin{rem}
During the writing of this paper, the paper by Berman \cite{Berm2} appeared in which more general results about subharmonicity of Ding-functional in the singular setting was proved. It seems very likely that his argument implies the following result.
\begin{conj}\label{conj1}
Let $\pi: (\mX, \mD)\rightarrow \C$ be a special degeneration for
$(X, D)$. Suppose the central fiber $(\mX_0, \mD_0)$ admits a
singular K\"ahler-Einstein metric of cone angle $2\pi\beta$ along
$\mD_0$.  Then  the log-Ding functional
$F_{\omega, (1-\beta)D}$ is bounded below. As a consequence, the
log-Mabuchi functional $\mM_{\omega, (1-\beta)D}$ is also bounded below.
\end{conj}
\end{rem}

\section{K\"ahler-Einstein metrics on
$X=\mathbb{P}^2$ singular along a conic}\label{P2case}
\subsection{Proof of Theorem \ref{conic}}\label{proofconic}

We first apply Phong-Song-Sturm-Weinkove's properness result in
Theorem \ref{PSSW} to show that the Ding-energy
$F^{\mathbb{P}^2}_{\omega}$ is proper on the space of
$SO(3)$-invariant K\"{a}hler metrics. For this, we need to show that
the centralizer of $SO(3,\mathbb{R})$ in $SU(3)$ is finite. Indeed,
if $\gamma\in \rm{Centr}_{SO(3,\mathbb{R})}SU(3)$, then $\gamma\cdot
C$ is a degree 2 curve invariant under $SO(3,\mathbb{R})$. But
there is a unique curve invariant under $SO(3,\mathbb{R})$ which is
just $C$ itself. So $\gamma\cdot C=C$ and we conclude $\gamma\in
SO(3,\mathbb{R})$. Since the center of $SO(3,\mathbb{R})$ is finite,
so the conclusion follows.

By the calculation in Example \ref{deg2slope}, we see that
$(\mathbb{P}^2,(1-\beta)D)$ is unstable when $0<\beta<1/4$, so there
is no conical K\"ahler-Einstein metric for $\beta\in (0, 1/4)$, by
Corollary \ref{mabuchimin} and Proposition \ref{Mabuchistable}. For the case $\beta=1/4$ we see the deformation to the normal cone considered in Proposition \ref{defnormalcone} shows that $(\mathbb P^2, 3/4D)$ is not log-K-polystable, so we can conclude the nonexistence of conical K\"ahler-Einstein metrics for $\beta=1/4$ by appealing to the more general result of Berman \cite{Berm2}(see Remark \ref{Bermanstability}).

To prove the existence for all $\beta\in (1/4,1]$, by Proposition
\ref{interpolate}, we only need to show the lower boundedness of
log-Mabuchi-energy when $\beta=1/4$. To do this,
 we  construct a special degeneration to conical K\"ahler-Einstein variety and apply Theorem \ref{specialdeg}. The special degeneration comes from deformation to the normal cone.
Let $\tilde{\mX}$ be the blow up of $\mathbb{P}^2\times\mathbb{C}$
along $D\times\{0\}$. Choose the line bundle
$\mL_{3/2}:=\pi^*K_{\mathbb{P}^2}^{-1}-3/2 E$ where $E$ is the
exceptional divisor. Then $\mL_{3/2}$ is semi-ample and the map
given by the complete linear system $|k \mL_{3/2}|$ for k
sufficiently big contracts the $\mathbb{P}^2$ in the central fibre
and we get a special test configuration $\mX$ with central fibre
being the weighted projective space $\mathbb{P}(1,1,4)$. It inherits an orbifold K\"{a}hler-Einstein
metric from the standard Fubini-Study metric on $\mathbb{P}^2$ by
the quotient map $\mathbb{P}^2=\mathbb{P}(1,1,1)\rightarrow
\mathbb{P}(1,1,4)$ given by $(Z_0,Z_1,Z_2)\rightarrow
(Z_0,Z_1,Z_2^4)=:[W_0,W_1,W_2]$. (See example \ref{P(1,1,4)} in
section \ref{relSW} for a toric explanation) The induced orbifold
K\"{a}hler-Einstein metric is the same as the conical
K\"{a}hler-Einstein metric on $\mathbb{P}(1,1,4)$ singular along the divisor $[W_2=0]$ with cone angle $2\pi/4$.
There is one orbifold singularity on $\mathbb{P}(1,1,4)$ of type
$\frac{1}{4}(1,1)$ as explained in example \ref{1/4(1,1)}. But this
does not cause any difficulty by the discussion in Section
\ref{singkahler}. So by Theorem \ref{specialdeg}, we get that the
log-Ding-energy $F^{\mathbb{P}^2}_{\omega,3/4 D}$ is bounded from
below. So by Proposition \ref{interpolate}, $F_{\omega,
(1-\beta)D}(\phi)$ is proper for $\beta\in (1/4,1]$ on the space of
$SO(3,\mathbb{R})$ invariant conical metrics. So by the existence
theorem explained in section \ref{existence}, we get the existence
of conical K\"{a}hler-Einstein metric on $(\mathbb{P}^2,
(1-\beta)D)$ for any $\beta\in (1/4,1]$.

\begin{rem}\label{coverP2}
When $\beta=1/2$, there exists an orbifold metric on
$(\mathbb{P}^2,1/2D)$ coming from the branched covering map given by
\begin{eqnarray*}
p: \mathbb{P}^1\times\mathbb{P}^1&\rightarrow& \mathbb{P}^2\\
{[}U_0,U_1{]},[V_0, V_1]&\mapsto &[U_0 V_0+U_1V_1, i(U_0V_1+U_1V_0),
i(U_0V_0-U_1V_1)]
\end{eqnarray*}
This is a degree 2 cover branching along the diagonal
$\triangle=\{[U_0,U_1],[U_0,U_1]\}\subset
\mathbb{P}^1\times\mathbb{P}^1$. Note that
\[
p(\triangle)=D=\{Z_0^2+Z_1^2+Z_2^2=0\}\subset \mathbb{P}^2.
\]
$Aut^0(\mathbb{P}^1\times\mathbb{P}^1)=PSL(2,\mathbb{C})\times
PSL(2,\mathbb{C})$. $PSL(2,\mathbb{C})$ acts diagonally on
$\mathbb{P}^1\times\mathbb{P}^1$ and we will denote such action by
$PSL(2)^{\rm diag}$. The map induces a morphism between groups:
\begin{eqnarray*}
\varphi: PSL(2, \mathbb C)^{\rm diag}&\longrightarrow & PSO(3,\mathbb{C})\\
\left(\begin{array}{cc} a &b \\c&d \end{array}
\right)&\mapsto&\frac{1}{ad-bc}\left(
\begin{array}{ccc}
(a^2+b^2+c^2+d^2)/2&-i(ab+cd)& -i(a^2-b^2+c^2-d^2)/2\\
i(ac+bd)&ad+bc&ac-bd\\
i(a^2+b^2-c^2-d^2)/2&ab-cd&(a^2-b^2-c^2+d^2)/2
\end{array}
\right)
\end{eqnarray*}
Note that $\varphi(SU(2))=SO(3,\mathbb{R})$. Actually this morphism
can be seen as the complexification of the covering map from $SU(2)$
to $SO(3,\mathbb{R})$.

\end{rem}

From the above proof it is tempting to expect that

\begin{conj} \label{conj2}The conical K\"ahler-Einstein metric $\omega_{\beta}$ on $\cP^2$ with cone angle $2\pi \beta$ along a smooth degree $2$ curve  converge in the Gromov-Hausdorff sense to the standard orbifold K\"ahler-Einstein metric on $\cP(1,1,4)$ as $\beta$ tends to $1/4$.
\end{conj}

Actually, more generally, assume there is a special degeneration
$(\mX,\mY)$ of the pair $(X,Y)$, such that $(\mX_0,\mY_0)$ is a
conical K\"{a}hler-Einstein pair. Then we expect $(X,Y)$ converges
to $(\mX_0,\mY_0)$ in Gromov-Hausdorff sense along certain
continuity method (either by the classical continuity method by
increasing Ricci curvature (cf. \cite{BM87},\cite{Li11b}), or by
changing cone angles(cf. \cite{Do11}) , or even by
log-K\"{a}hler-Ricci flow(cf. in \cite{SuWa}). This philosophy is
certainly  well known to the expects in the field. In particular,
this is related to \cite{Ti10} and \cite{DoSu}.

\begin{rem}
 In \cite{HaPr}, $\Q$-Gorenstein smoothable degenerations of $\mathbb{P}^2$ are classified. They are given by partial smoothings of weighted projective planes $\mathbb{P}^2(a^2,b^2,c^2)$ where
$(a,b,c)$ satisfies the Markov equation: $a^2+b^2+c^2=3abc$. Different solutions are related by an operation called mutation: $(a,b,c)\rightarrow (a,b,3ab-c)$.
The first several solutions are $(1,1,1)$, $(1,1,2)$, $(1,5,2)$, $(1,5,13)$, $(29,5,2)$. The above construction gives first geometric realization of such degeneration corresponding to
the mutation $(1,1,1)\rightarrow (1,1,2)$. We expect there is similar geometric realization of every mutation.
\end{rem}

\subsection{Calabi-Yau cone metrics on three dimensional $A_2$ singularity}
Through a stimulating discussion with Dr. Hans-Joachim Hein, we
learned that Theorem \ref{conic} has the following application.
Recall it was discovered by Gauntlett-Martelli-Sparks-Yau
\cite{GMSY} that there may not exist Calabi-Yau cone metrics on
certain isolated quasi-homogeneous hypersurface singularities, with
the obvious Reeb vector field. In particular, there are two
constraints: Bishop obstruction and Lichnerowicz obstruction. As an
example, the case of three dimensional $A_{k-1}$ singularities was
studied. Recall a three dimensional $A_{k-1}$ singularity is the
hypersurface in $\C^4$ defined by the following equation
$$x_1^2+x_2^2+x_3^2+x_4^{k}=0. $$
There is a standard Reeb vector field $\xi_k$ which generates the $\C^*$ action with weights $(k, k, k, 2)$.  Let $L_k$ be the Sasaki link of the $A_{k-1}$ singularity. Then the existence of a Calabi-Yau cone metric with Reeb vector field $\xi_k$ is equivalent to the existence of a  Sasaki-Einstein metric on $L_k$. In \cite{GMSY}, using the Bishop obstruction, it was proved that $L_k$ admits no Sasaki-Einstein metric  for $k>20$, and using Lichnerowicz obstruction this bound was improved to $k>3$. For $k=2$ this is the well-known conifold singularity  and there is a homogeneous Sasaki-Einstein metric on the link $L_2$. For $k=3$ by Matsushima's theorem the possible Sasaki-Einstein metric on $L_3$ must be invariant under $SO(3;\R)$ action, and is of cohomogeneity one. The ordinary differential equation has been written down explicitly in \cite{GMSY}, and it is an open question in \cite{GMSY} whether $L_3$ admits a Sasaki-Einstein metric. \\

In the language of Sasaki geometry, the above examples $L_k$ are all
quasi-regular, meaning that the Reeb vector field $\xi_k$ generates
an  $S^1$ action on $L_k$, and the quotient $M_k$ is a polarized
orbifold $M_k$(in the sense of \cite{RoTh2}). The existence of a
Sasaki-Einstein metric on $L_k$ is equivalent to the existence of an
orbifold K\"aler-Einstein metric on $M_k$. In the above concrete
cases, the orbifold $M_k$ is the hypersurface in $\cP(k, k, k, 2)$
defined by the same equation $x_1^2+x_2^2+x_3^2+x_4^{k}=0$. Note
that $\cP(k,k,k,2)$ is not well-formed. When $k=2m+1$ is odd, then
\begin{eqnarray*}
\cP(2m+1,\!2m+1,\!2m+1,\!2)&\stackrel{\cong}{\longrightarrow}&
\cP(2m+1,\!2m+1,\!2m+1,\!2(2m+1))=\cP(1,1,1,2)\\
{[ x_1,x_2,x_3,x_4 ]}&\mapsto& { [ x_1,x_2,x_3,x_4^{2m+1} ]}.
\end{eqnarray*}
When $k=2m$ is even, then
\begin{eqnarray*}
\cP(2m,2m,2m,2)=\cP(m,m,m,1)&\stackrel{\cong}{\longrightarrow}&\cP(m,m,m,m)=\cP(1,1,1,1)\\
{[x_1,x_2,x_3,x_4]}&\mapsto&{[x_1,x_2,x_3,x_4^{m}]}.
\end{eqnarray*}
So $M_k$ is isomorphic to $\{z_1^2+z_2^2+z_3^2+z_4=0\}\cong\cP^2$
when $k$ is odd, and to
$\{z_1^2+z_2^2+z_3^2+z_4^2=0\}\cong\cP^1\times\cP^1$ when $k$ is
even.

Regarding the non-well-formed orbifold structure it is not hard to
see that when $k$ is odd we get $(\cP^2, (1-1/k)D)$ and when $k$ is
even we get $(\cP^1\times\cP^1, (1-2/k)\Delta)$.  Thus we see the
close relationship between the existence of Sasaki-Einstein metric
on $L_k$ and the existence of conical K\"ahler-Einstein metric on
$(\cP^2, (1-1/k)D)$. In particular we know there is no
Sasaki-Einstein metric on $L_k$ for $k>3$ by Example
\ref{deg2slope},. This is not surprising at all, since by
\cite{RoTh2} the Lichnerowicz obstruction could be interpreted as
slope stability for orbifolds. The new observation here is the case
$k=3$ follows from Theorem \ref{conic}. So we know the three
dimensional $A_2$ singularity admits a Calabi-Yau cone metric with
the standard Reeb vector field. This is Corollary \ref{CYcone}.

The corresponding Sasaki-Einstein metric on $L_k$ is invariant under
the $SO(3;\R)$ action. It would be interesting to find an explicit
solution by solving the ODE written in \cite{GMSY}.   In \cite{Conti}
cohomogeneity one Sasaki-Einstein five manifolds were classified,
but the above result suggests that the classification is incomplete.

\begin{rem}\label{numerical}
 In \cite{Li12b}, the first author used numerical method to solve the ODE in \cite{GMSY}. The numerical
 results confirm our theoretical result. Moreover, numerical results show that Conjecture \ref{conj2} is true. Also,
 by calculating the simplest examples, one finds there are indeed cases which were ignored in \cite{Conti}. For details,
 see \cite{Li12b}.
\end{rem}

\section{K\"ahler-Einstein metrics from branched cover}\label{brcover}

One of our motivation for this paper is to construct smooth
K\"{a}hler-Einstein metrics using branch covers(see \cite{AGP}, \cite{GK} for such kind of constructions). If $D\sim m D_1$
with $D_1$ being an integral divisor, we can construct branch cover of
$X$ with branch locus $D$.
\begin{eqnarray*}
B&\subset&Y\\
\downarrow&&\downarrow\pi \\
 D&\subset& X
\end{eqnarray*}
The canonical divisors of $X$ and $Y$ are related by Hurwitz formula:
\[
K_Y=\pi^*(K_X+\frac{m-1}{m}D)
\]
In our setting, since $D\sim -\lambda K_X$, we get
\begin{equation}\label{Hurwitz}
K_Y^{-1}=\left(1-\frac{m-1}{m}\lambda\right)\pi^*K_X^{-1}=r(1/m)\pi^*K_X^{-1}
\end{equation}
We have the following 3 cases to consider.
\begin{enumerate}
\item(Positive Ricci) $-(K_X+(1-1/m)D)$ is ample. This is equivalent
to $r(1/m)>0$.

Example: $X=\mathbb{P}^2$. Define $\deg Y=(K_Y^{-1})^2$.
\begin{itemize}
\item
$\deg D=2$, $m=2$, $\lambda=2/3$. $\deg Y=4$.
$Y=\mathbb{P}^1\times\mathbb{P}^1$.
\item
$\deg D=3$, $m=3$, $\lambda=1$. $\deg Y=3$. $Y$ is a cubic surface.
\item
$\deg D=4$, $m=2$, $\lambda=4/3$. $\deg Y=2$. $Y$ is
$Bl_{7}\mathbb{P}^2$.
\end{itemize}

\item(Ricci flat) $K_X+(1-1/m)D\sim 0$. This is equivalent to
$r(1/m)=0$.

Example: $X=\mathbb{P}^2$.
\begin{itemize}
\item $\deg D=4$, $m=4$, $\lambda=4/3$. $Y$ is a K3 surface in $\mathbb{P}^3$.
\item $\deg D=6$, $m=2$, $\lambda=2$. $Y$ is a K3 surface.
\end{itemize}
\item(Negative Ricci) $K_X+(1-1/m)D$ is ample.

Example: $X$ is $\mathbb{P}^2$ and $D$ is a general smooth, degree
$d$ curve such that $\lambda=d/3$. Choose $m|d$. Except for the
cases already listed above, $K_X+(1-1/m)D$ is ample.
%and there exists
%an orbifold K\"{a}hler-Einstein metric with negative Ricci curvature.
\end{enumerate}
Assume we have already constructed an orbifold K\"{a}hler-Einstein
metric $\widehat{\omega}_{KE}$ on $(X,(1-1/m)D)$. Then
$\pi^*\widehat{\omega}_{KE}$ is a smooth K\"{a}hler-Einstein metric
on $Y$. Note that orbifold K\"{a}hler metric can be seen as a
special case of conical K\"{a}hler metric, i.e. when the cone angle is
equal to $2\pi/m$ for some $m\in\mathbb{Z}$. So existence of conical
K\"{a}hler-Einstein metrics with angle $2\pi/m$ will give rise to
smooth K\"{a}hler-Einstein metrics.
Using the existence theory for conical K\"ahler-Einstein metrics, we can construct a lot of
smooth K\"{a}hler-Einstein metrics on Fano manifolds using branch
covers. More precisely, using the notation of branch-covering above, we have
\begin{thm}\label{KEcover}
If there is conical K\"{a}hler-Einstein metric on $(X,(1-1/m)D)$, then
there is a smooth K\"{a}hler-Einstein metric on $Y$. In particular,
if $X$ admits a K\"{a}hler-Einstein metric and $\lambda\ge 1$, then
there exists smooth K\"{a}hler-Einstein metric on $Y$.
\end{thm}

%\subsection{Construct smooth K\"{a}hler-Einstein metrics on branch
%covers}\label{branchcover}

%If $D=m D_1$, we can construct branch cover $X$ with branch locus
%$D$.
%\begin{eqnarray*}
%B&\subset&Y\\
%\downarrow&&\downarrow\pi \\
% D&\subset& X
%\end{eqnarray*}

To begin the proof, we first observe the following
\begin{prop}\label{coverfunc}
Fix an orbifold K\"{a}hler metric $\omega$ on $(X,(1-1/m)D)$. The branch cover
$\pi$ induces a map from $\mathcal{PSH}(\omega)$ to
$\mathcal{PSH}(\pi^*\omega)$ by pulling back. The energy functionals
are compatible with this pull back.
\[
F^{Y}_{r(1/m)\pi^*\omega}(r(1/m)\pi^*\phi)=mF^{X}_{\omega,(1-1/m)D}(\phi),
\]
\[
\mathcal{M}^{Y}_{r(1/m)\pi^*\omega}(r(1/m)\pi^*\phi)=m\mathcal{M}^X_{\omega,(1-1/m)D}(\phi).
\]
Similar relation holds for the functionals $F^0_\omega(\phi)$, $I$
and $J$.
\end{prop}
\begin{proof}
For any orbifold K\"{a}hler metric  $\omega\in 2\pi c_1(X)$, there
exists $H_{\omega,(1-1/m)D}$ such that
\begin{equation}\label{orbifoldH}
Ric(\omega)-r(1/m)\omega-(1-1/m)\{D\}=\sddbar H_{\omega,(1-1/m)D}.
\end{equation}
$\tilde{\omega}=r(1/m)\pi^*\omega$ is a smooth K\"{a}hler metric in
$c_1(Y)$ (see \eqref{Hurwitz}). Note that $\omega^n$ has poles along $D$, but  $\pi^*\omega^n$ is a smooth volume form. From \eqref{orbifoldH}, we get
\[
Ric(\tilde{\omega})-\tilde{\omega}=\sddbar
\pi^*H_{\omega,(1-1/m)D}.
\]
So $h_{\tilde{\omega}}:=H_{\tilde{\omega},0}=\pi^*H_{\omega,(1-1/m)D}$ and
$e^{h_{\tilde{\omega}}}\tilde{\omega}^n=\pi^*(e^{H_{\omega,(1-1/m)D}}\omega^n)$.
\[
\int_X
e^{H_{\omega,(1-\beta)D}-r(1/m)\phi}\omega^n/n!=\frac{1}{m}\int_Y
e^{h_{\tilde{\omega}}-\pi^*(r(1/m)\phi)}\tilde{\omega}^n/n!.
\]
So we get the identity for log-Ding-energy on $X$ and $Y$. Similarly, by the defining formula for the $F_\omega^0(\phi)$, $I$, $J$ functional in
Definition \ref{defF0IJ}, the relation stated in the proposition
holds.
\end{proof}
\begin{proof}[Proof of Theorem \ref{KEcover}]
We can choose the reference metric $\omega$ on $X$ to be orbifold
metric. Then the pull back $\tilde{\omega}=r(1/m)\pi^*\omega$ is a smooth
K\"{a}hler metric on $Y$. If $\omega_{KE}=\omega+\sddbar\phi_{KE}$
is the conical K\"{a}hler-Einstein metric on $(X,(1-1/m)D,c_1(X))$,
then $\tilde{\phi}_{KE}=r(1/m)\pi^*\phi_{KE}$ is the bounded solution of
the following Monge-Amp\`{e}re equation on $Y$.
\[
(\tilde{\omega}+\sddbar\tilde{\phi})^n=e^{h_{\tilde{\omega}}-\tilde{\phi}}\tilde{\omega}^n.
\]
Because locally, the covering map is given by $z\rightarrow z^m$, by
asymptotic expansion obtained in \cite[Proposition 4.3]{JMRL}, one
verifies $\tilde{\phi}_{KE}$ is $C^2$. So by elliptic theory,
$\tilde{\phi}_{KE}$ is indeed a smooth solution of
K\"{a}hler-Einstein equation on $Y$.
\end{proof}

\section{Convergence of conical KE to smooth KE}\label{convke}
In this section, we prove the convergence statement in Corollary
\ref{maincor} and related discussions following it. So we assume there exists smooth K\"{a}hler-Einstein
metric on $X$. When $Aut(X)$ is discrete, then $\omega_{KE}$ is
invariant under $Aut(X)$. In this case, the Mabuchi energy is proper
on $\hat{\mathcal{H}}(\omega)$.
\begin{thm}\label{autdiscrete}
Assume $\omega_\beta=\omega+\sddbar\phi_\beta$, and
$\omega_{KE}=\omega+\sddbar\phi_{KE}$ then, $\phi_\beta$ converges
to $\phi_{KE}$ in $C^0$-norm. Moreover, $\phi_\beta$ converges
smoothly on any compact set away from $D$.
\end{thm}
\begin{proof}
By Theorem \ref{interpolate}, the log-Mabuchi-energy
$\mathcal{M}_{\omega,(1-\beta)D}$ is proper on $\hat{\mathcal{H}}$
for $\beta\in (0,1]$. Furthermore, there exists a constant $C$
independent of $\beta$ such that
\begin{equation}\label{uniformproper}
\mathcal{M}_{\omega,(1-\beta)D}(\omega_\phi)\ge
C_1I_\omega(\omega_\phi)-C_2.
\end{equation}
When $\omega_\phi=\omega_\beta$ is the conical K\"{a}hler-Einstein
metric on $(X,(1-\beta)D)$ then
\[
\mathcal{M}_{\omega,(1-\beta)D}(\omega_\beta)\le
\mathcal{M}_{\omega,(1-\beta)D}(\omega)=0.
\]
So from \eqref{uniformproper} , we see that there exists a constant
$C$ independent of $\beta$ such that
\[
I_\omega(\omega_\beta)\le C.
\]
Assume $\omega_\beta=\omega+\sddbar\phi_\beta$. By Proposition
\ref{c0byint}(\cite{JMRL}), there exists a constant $C$ independent
of $\beta$ such that
\[
Vol(X)\cdot Osc(\phi_\beta)\le
I_\omega(\omega_\beta)+C=\int_X\phi(\omega^n-\omega_\beta^n)/n!+C.
\]
So $\|\phi_\beta\|_{C^0}$ is uniformly bounded. Now the theorem
follows from standard pluripotential theory. For the last statement,
we use the same argument as in the in the proof of existence result
in section \ref{existence}. First we use Chern-Lu's inequality:
\[
\Delta_{\beta}(\log tr_{\omega_\beta}\omega-C\phi_\beta)\ge
(C_1-\lambda n)+(\lambda-C_2 tr_{\omega_\beta}\omega).
\]
to get $C^2$-estimate on $\phi_\beta$. Note that it's easy to verify
from the calculation in the Appendix of \cite{JMRL} that we can
choose the upper bound on bisectional curvature to be independent of
$\beta$ at least when $\beta\ge \delta>0$. Then we can use Krylov-Evan's estimate to get uniform
higher order estimate on any compact set away from $D$. The smooth
convergence follows from these uniform estimates.
\end{proof}

When $Aut(X)$ is continuous, then $Aut(X)$ is the complexification
of $G:=Isom(X,\omega_{KE})$. By \cite{BM87}, the moduli space of
K\"{a}hler-Einstein metrics (denoted by $\mathscr{M}_{KE}$) is
isomorphic to the symmetric space $G^{\mathbb{C}}/G $. So
\[
T_{\omega_{KE}}\mathscr{M}_{KE}=\mathfrak{g}=Lie G.
\]
Recall that
\[
\mathfrak{g}=(\Lambda_1^{\mathbb{R}})_0=\{\theta\in C^{\infty}(X); (\Delta_{KE}+1)\theta=0, \int_X\theta\omega_{KE}^n/n!=0\}.
\]
Now we want to identify the limit $\omega_{KE}^D$ as $\beta\rightarrow 1$. $\omega_{KE}^D$
is the critical point of the following functional, which is part of
log-Mabuchi-functional.
\begin{lem}
Define the functional
\[
\mathscr{F}_{\omega,D}(\omega_\phi)=\lambda(I-J)_{\omega}(\omega_\phi)+\int_X\log|s|_h^2(\omega_\phi^n-\omega^n)/n!
\]
where $\lambda\omega=-\sddbar \log|\cdot|_h^2$. Then $\mathscr{F}$
satisfies the following properties:
\begin{enumerate}
\item
\begin{equation}\label{partmabuchi}
\mathcal{M}_{\omega,(1-\beta)D}(\omega_\phi)=\mathcal{M}_{\omega}(\omega_\phi)+(1-\beta)\mathscr{F}_{\omega,D}(\omega_\phi).
\end{equation}
\item
$\mathscr{F}_{\omega}$ satisfies the cocycle condition. More
precisely, for $\phi, \psi\in\mathcal{PSH}_\infty(\omega)$, we have
\begin{eqnarray*}
&&\mathscr{F}_{\omega,D}(\omega_\phi)-\mathscr{F}_{\omega_\psi,D}(\omega_\phi)=\mathscr{F}_{\omega,D}(\omega_\psi),\\
&&\mathscr{F}_{\omega}(\omega_\psi)=-\mathscr{F}_{\omega_\psi}(\omega).
\end{eqnarray*}
\item
$\mathscr{F}$ is convex along geodesics of K\"{a}hler metrics. .
\end{enumerate}
\end{lem}

\begin{proof}
The first item follows from the expression for log-Mabuchi energy
in, for example, formula \eqref{linearmabuchi}. The second statement
follows from the cocycle properties of
$\mathcal{M}_{\omega,(1-\beta)D}$ and $\mathcal{M}_\omega$. It can
also be verified by direct calculations.
For the last statement, it is well known that $\mathscr{M}$ is a
totally geodesic submanifold of the space of smooth K\"{a}hler
metrics in $c_1(X)$ and $(I-J)_\omega(\omega_\phi)$ is convex on the
space of smooth K\"{a}hler metrics. Assume $\phi(t)$ is a geodesic,
i.e. $\ddot{\phi}-|\nabla\dot{\phi}|_{\omega_\phi}^2=0$.
\begin{eqnarray*}
\frac{d}{dt}(I-J)_\omega(\omega_\phi)&=&-\int_X\phi\Delta_{\omega_\phi}\dot{\phi}\omega_\phi^n/n!
=-\int_X\dot{\phi}(\omega_\phi-\omega)\wedge\omega_\phi^{n-1}/(n-1)!\\
&=&n\frac{d}{dt}F_\omega^0(\phi)+\int_X\dot{\phi}\omega\wedge\omega_\phi^{n-1}/(n-1)!.
\end{eqnarray*}
\begin{eqnarray*}
\frac{d}{dt}\int_X\log|s|_h^2(\omega_\phi^n-\omega^n)/n!&=&\int_X\log|s|_h^2\Delta_{\omega_\phi}\dot{\phi}\omega_\phi^n/n!=\int_X(-\lambda\omega+2\pi\{D\})\dot{\phi}\omega_\phi^{n-1}/(n-1)!.
\end{eqnarray*}
So combining the above two identities, we get
\[
\frac{d}{dt}\mathscr{F}_{\omega,D}(\omega_\phi)=\frac{d}{dt}(n\lambda
F_\omega^0(\phi)-F_{\omega,2\pi D}^0(\phi)).
\]
This is certainly true by the way how we integrate the log-Futaki
invariant to get the log Mabuchi energy. Now since
$F_\omega^0(\phi)$ is affine along geodesics of K\"{a}hler metrics.
\begin{eqnarray}\label{2nder}
\frac{d^2}{dt^2}\mathscr{F}_{\omega,2\pi
D}(\omega_\phi)&=&-\frac{d^2}{dt^2}F_{\omega,D}^0(\phi)=\frac{\int_{2\pi
D}(\omega+\partial\bar{\partial}\Phi)^{n}/n!}{dt\wedge
d\bar{t}}\nonumber\\
&=&\int_{2\pi D}\ddot{\phi}\omega_{KE}^{n-1}/(n-1)!-\int_{2\pi D}\partial\dot{\phi}\wedge\bar{\partial}\dot{\phi}\wedge\omega_{KE}^{n-2}/(n-2)!\nonumber\\
&=&\int_{2\pi D}(|\nabla\dot{\phi}|_{\omega_{KE}}^2-|\nabla^D\dot{\phi}|^2_{\omega_{KE}|_D})\omega_{KE}^{n-1}/(n-1)!\nonumber\\
&=&\int_{2\pi
D}|(\nabla\dot{\phi})^{\perp}|_{\omega_{KE}}^2\omega_{KE}^{n-1}/(n-1)!\ge
0.
\end{eqnarray}
\end{proof}
\begin{lem}
We have the following different formulas for the Hessian of
$\mathscr{F}_{\omega,D}$ on $\mathscr{M}_{KE}$.
\begin{eqnarray}
Hess\mathscr{F}(\theta,\theta)&=&\int_{2\pi D}|\nabla\theta^{\perp}|^2\omega_{KE}^{n-1}/(n-1)!\label{hess1}\\
&=&\lambda\int_X\theta^2\omega_{KE}^n/n!+\int_X(-\theta^2+\theta^i\theta_i)(\lambda\phi-\log|s|_h^2)\omega_{KE}^n/n!\nonumber\\%.\label{hess2}
&=&\lambda\int_X\theta^2\omega_{KE}^n/n!+\int_X(\theta^2-\theta^i\theta_i)(\log|s|_{he^{-\lambda\phi}}^2)\omega_{KE}^n/n!.\label{hess2}
\end{eqnarray}
\end{lem}
\begin{proof}
The first identity follows from \eqref{2nder} because $\theta=\frac{\partial\phi}{\partial t}|_{t=0}$.
%For the second identity, we shall give two calculations.
Let's prove the 2nd identity. For $\theta\in\Lambda_1=Ker(\Delta_{KE}+1)$,
$\nabla\theta$ is a holomorphic vector field generating a one
parameter subgroup $\sigma_t$ in $Aut(X)$. Let
$\sigma_t^*\omega_{KE}=\omega_{KE}+\sddbar\phi_t$. Then $\phi_t$
satisfies the geodesic equation:
$\ddot{\phi}-|\nabla\dot{\phi}|_{\omega_\phi}^2=0$ with initial
velocity $\frac{d}{dt}\phi|_{t=0}=\theta$.
%\begin{itemize}
%\item(calculation 1)
%\begin{eqnarray*}
%\frac{d^2}{dt^2}\mathscr{F}_{\omega,D}&=&\frac{d^2}{dt^2}F^0_{\omega,D}(\phi)=\int_{2\pi D}\ddot{\phi}\omega_{KE}^{n-1}/(n-1)!-\int_{2\pi D}\partial\dot{\phi}\wedge\bar{\partial}\dot{\phi}\wedge\omega_{KE}^{n-2}/(n-2)!\\
%&=&\int_X\left(\sddbar\log|s|_h^2+\lambda(-\sddbar{\phi}+\omega_{KE})\right)\wedge(\ddot{\phi}\omega_{KE}-(n-1)\partial\dot{\phi}\wedge\bar{\partial}\dot{\phi})\wedge\omega_{KE}^{n-2}/(n-1)!\\
%&=&\int_X(\log|s|_h^2-\lambda\phi)(\sddbar\ddot{\phi}\wedge\omega_{KE}^{n-1}+(n-1)\partial\bar{\partial}\dot{\phi}\wedge\partial\bar{\partial}\dot{\phi}\wedge\omega_{KE}^{n-2})/(n-1)!+\\
%&&+\lambda\int_X(\ddot{\phi}\omega_{KE}^n-(n-1)\partial\dot{\phi}\wedge\bar{\partial}\dot{\phi}\wedge\omega_{KE}^{n-1})/(n-1)!\\
%&=&\int_X(\log|s|_h^2-\lambda\phi)(\Delta|\nabla\theta|^2+(\Delta\theta)^2-\theta_{i\bar{j}}\theta^{i\bar{j}})\omega_{KE}^n/n!
%+\lambda\int_X\theta^i\theta_i\omega_{KE}^n/n!+\\
%&&+\lambda\int_X(\ddot{\phi}-|\nabla\dot{\phi}|_{\omega_{KE}}^2)\omega_{KE}^n/(n-1)!\\
%&=&\int_X(\log|s|_{he^{-\lambda\phi}}^2)(\theta^2-\theta^i\theta_i)+\lambda\int_X\theta^2\omega_{KE}^n/n!.
%\end{eqnarray*}
%\item(Calculation 2)
\begin{eqnarray*}
\left.\frac{d^2}{dt^2}\right|_{t=0}\omega_t^n&=&\frac{d}{dt}(\Delta\dot{\phi}\omega_t^n)=(\dot{\Delta}\dot{\phi}+\Delta\ddot{\phi}+(\Delta\dot{\phi})^2)\omega_{KE}^n\\
&=&(-\theta^{i\bar{j}}\theta_{i\bar{j}}+(\theta^i\theta_i)_j^{\phantom{j}j}+\theta^2)\omega_{KE}^n\\
&=&(-\theta^{i\bar{j}}\theta_{i\bar{j}}+\theta^i_{\phantom{i}j}\theta_i^{\phantom{i}j}+\theta^{i\phantom{j}j}_{\phantom{i}j}\theta_i+\theta^2)\omega_{KE}^n/n!\\
&=&(\theta^2-\theta^i\theta_i)\omega_{KE}^n
\end{eqnarray*}
Note that in the last identity, the relation $\Delta\theta=\theta_i^{\phantom{j}i}=-\theta$ was used. So we get
\begin{eqnarray*}
Hess\mathscr{F}(\theta,\theta)&=&\frac{d^2}{dt^2}\left(\lambda(I-J)_\omega(\omega_t)+\int_X\log|s|_h^2(\omega_t^n-\omega^n)/n!\right)\\
&=&-\lambda\int_X\dot{\phi}\Delta\dot{\phi}\omega_{KE}^n/n!+\int_X(-\lambda\phi+\log|s|_h^2)\frac{d^2}{dt^2}\omega_t^n/n!\\
&=&\lambda\int_X\theta^2\omega_{KE}^n/n!+\int_X(\log|s|_{he^{-\lambda\phi}}^2)(\theta^2-\theta^i\theta_i)\omega_{KE}^n/n!.
\end{eqnarray*}
%\end{itemize}
\end{proof}
\begin{lem}\label{nondeg}
If there is no holomorphic vector field on $X$ which is tangent to
$D$, i.e. $Aut(X,D)$ is discrete, then $Hess\mathscr{F}$ is
non-degenerate at any point $\omega_{KE}\in\mathscr{M}_{KE}$. In
particular, this holds when $\lambda\ge 1$.
\end{lem}
\begin{proof}
We have seen $Hess\mathscr{F}$ is non-negative at any point
$\omega_{KE}\in\mathscr{M}_{KE}$ using formula \eqref{hess1}.
$Hess\mathscr{F}$ is degenerate if and only if $\int_{2\pi
D}|(\nabla\theta)^{\perp}|^2\omega_{KE}^{n-1}/(n-1)!=0$. This
happens if and only if $(\nabla\theta)^{\perp}\equiv0$ on $D$, i.e.
when $\nabla\theta$ is tangent to $D$. The last statement follows
from Corollary \ref{nohv} (see also \cite{SW}).
\end{proof}
\begin{lem}\label{chooseke}
When restricted to $\mathscr{M}_{KE}$, there exists a unique minimum $\omega_{KE}^{D}$ of $\mathscr{F}_{\omega,D}(\omega_\phi)$.
\end{lem}
\begin{proof}
By the previous Lemma, $\mathscr{F}_{\omega,D}$ is a convex functional on the space $\mathscr{M}_{KE}\cong G^{\mathbb{C}}/G$. To prove the existence of critical point, we only need to show it's proper on $G^{\mathbb{C}}/G$.   Because we assumed $\lambda\ge 1$ and there exists K\"{a}hler-Einstein on $X$, by Theorem
\ref{mabuchimin}, $\mathcal{M}_{\omega,(1-\beta)D}$ is proper for $\beta\in (0,1)$. Because the Mabuchi energy is constant on $\mathscr{M}_{KE}$, by equality \eqref{partmabuchi},  $\mathscr{M}_{\omega,(1-\beta)D}=(1-\beta)\mathscr{F}_{\omega,D}$+constant is proper on $\mathscr{M}_{KE}$.
\end{proof}
Write $\omega_{KE}^D=\omega+\sddbar\phi_{KE}^D$, then it satisfies the critical point equation
\[
\int_X(\log|s|_h^2-\lambda\phi_{KE}^D)\psi(\omega_{KE}^D)^n/n!=0.
\]
for any
$T_{\omega_{KE}^D}\mathscr{M}_{KE}\cong\Lambda_1(\omega_{KE}^D)/\mathbb{C}$.
In other words, $\lambda\phi_{KE}^D-
\log|s|_h^2\in\Lambda_1^{\perp}$.
\begin{prop}
As $\beta\rightarrow 1$, the conical K\"{a}hler-Einstein metrics
$\omega_{\beta}$ converges to a unique smooth K\"{a}hler-Einstein
metric $\omega_{KE}^D\in\mathscr{M}_{KE}$.
\end{prop}
\begin{proof}
Recall that the conical K\"{a}hler-Einstein equation can be written
as
\[
(\omega+\sddbar{\phi})^n=e^{h_\omega-r(\beta)\phi}\frac{\omega^n}{|s|^{2(1-\beta)}}.
\]
Any $\omega_{KE}=\omega+\sddbar\phi_{KE}\in\mathscr{M}_{KE}$ satisfies the equation
\[
(\omega+\sddbar\phi_{KE})^n=e^{h_\omega-\phi_{KE}}\omega^n.
\]
By Lemma \ref{chooseke}, there exists a unique minimum
$\omega_{KE}^D$ of the functional $\mathscr{F}_{\omega,D}$ on
$\mathscr{M}_{KE}$.  We will choose $\omega_{KE}=\omega_{KE}^D$ in
the following argument. Divide the above two equations to get
\[
\log\frac{(\omega+\sddbar\phi)^n}{(\omega+\sddbar\phi_{KE})^n}=\phi_{KE}-r(\beta)\phi-(1-\beta)\log|s|_{h}^2.
\]
Let $\phi=\phi_{KE}+\psi$ and $\psi=\theta+\psi'$ with $\theta\in \Lambda_1$ and $\psi'\in \Lambda_1^{\perp}$, then
\begin{equation}\label{kestart}
\log\frac{(\omega_{KE}+\sddbar(\theta+\psi'))^n}{\omega_{KE}^n}+r(\beta)(\theta+\psi')=(1-\beta)(\lambda\phi_{KE}-\log|s|_h^2).
\end{equation}
We use Bando-Mabuchi's bifurcation method to solve the equation for
$\beta$ close to 1. First project to $\Lambda_1^{\perp}$ to get
\begin{equation}\label{projection1}
(1-P_0)\left(\log\frac{(\omega_{KE}+\sddbar(\theta+\psi'))^n}{\omega_{KE}^n}\right)+r(\beta)\psi'=(1-\beta)(\lambda\phi_{KE}-\log|s|_h^2).
\end{equation}

The equation is satisfied for $(\beta,\psi,\theta)=(1,0,0)$. The
linearization of the left side of this equation with respect to
$\psi'$ is
\[
(1-P_0)(\tilde{\Delta}_\theta+r(\beta))\psi'
\]
where $\tilde{\Delta}_\theta$ is the Laplacian with respect to
$\omega_{KE}+\sddbar\theta$. Since $\Lambda_1=Ker(\Delta_{KE}+1)$,
there exists a positive constant $\delta>0$, such that
\[
(1-P_0)(-\Delta_{\omega_{KE}}-1)\ge \delta>0.
\]
By continuity, it's easy to see that
\[
(1-P_0)(-\tilde{\Delta}_\theta-r(\beta))\ge \delta/2>0.
\]
for $(\beta,\theta)$ close to $(1,0)$.  In other words, the inverse
of $(1-P_0)(\tilde{\Delta}_{\theta}+r(\beta))$ has uniformly bounded
operator norm for $(\beta,\theta)$ close to $(1,0)$. So by implicit
function theorem, there exists solution $\psi'_{\beta,\theta}$ for
$\beta$ near 1 and $\theta$ small. Now to solve the equation
\eqref{kestart}, we only need to solve the following equation,
obtained by projecting to $\Lambda_1$,
\begin{equation}\label{projection2a}
P_0\left(\log\frac{(\omega_{KE}+\sddbar(\theta+\psi'))^n}{\omega_{KE}^n}\right)=-r(\beta)\theta.
\end{equation}
To solve this, we need to take the gaugh group $G=Isom(X,\omega_{KE})$ into account and rewrite \eqref{projection2a} in another form. For any
$\sigma\in G$ near $Id$, we have a function $\theta=\theta_\sigma$ satisfying $\sigma^*\omega_{KE}=\omega_{KE}+\sddbar(\theta+\psi'_{1,\theta})$.
Because $\sigma^*\omega_{KE}$ is a smooth K\"{a}hler-Einstein metric, we have the equation
\[
\log\frac{(\omega_{KE}+\sddbar(\theta+\psi'_{1,\theta}))^n}{\omega_{KE}^n}=-(\theta+\psi'_{1,\theta})
\]
Now let $\psi_{\beta,\theta}'=\psi_{1,\theta}'+(1-\beta)\xi_{\beta,\theta}$. We can rewrite the equation \eqref{kestart} in the following form
\begin{eqnarray*}
&&\log\frac{(\omega_{KE}+\sddbar(\theta+\psi'_{1,\theta}+(1-\beta)\xi_{\beta,\theta}))^n}{\omega_{KE}^n}=\\
&=&-(1-\lambda(1-\beta))(\theta+\psi'_{1,\theta}+(1-\beta)\xi_{\beta,\theta})
+(1-\beta)(\lambda\phi_{KE}-\log|s|_h^2)\\
&=&\log\left(\frac{\omega_\theta^n}{\omega_{KE}^n}\right)+(1-\beta)\lambda(\theta+\psi'_{1,\theta})-(1-\beta)r(\beta)\xi_{\beta,\theta}+(1-\beta)(\lambda\phi_{KE}-\log|s|_h^2).
\end{eqnarray*}
where $\omega_\theta=\omega_{KE}+\sddbar(\theta+\psi'_{1,\theta})$.  In particular, it's easy to see that \eqref{projection1} is equivalent to
\[
\frac{1}{1-\beta}(1-P_0)\left(\log\frac{(\omega_{KE}+\sddbar(\theta+\psi'_{1,\theta}+(1-\beta)\xi_{\beta,\theta}))^n}{(\omega_{KE}+\sddbar(\theta+\psi'_{1,\theta}))^n}\right)=\lambda\psi_{1,\theta}'-r(\beta)\xi_{\beta,\theta}+(\lambda\phi_{KE}-\log|s|_h^2).
\]
Let $\beta\rightarrow1$ to get
\[
(1-P_0)\left((\Delta_\theta+1)\xi_{1,\theta}\right)-\lambda\psi'_{1,\theta}=\lambda\phi_{KE}-\log|s|_h^2.
\]
where $\Delta_\theta$ is the Laplacian with respect to the metric $\omega_\theta$. In particular, $\Delta_0=\Delta_{KE}$. Since $Im(\Delta_0+1)=(Ker(\Delta_0+1))^{\perp}=\Lambda_1^{\perp}$, so in particular,
\begin{equation}\label{xi10eq}
(\Delta_0+1)\xi_{1,0}=\lambda\phi_{KE}-\log|s|_h^2.
\end{equation}
Now the equation \eqref{projection2a} is equivalent to
\begin{equation}\label{projection2b}
P_0\left(\frac{1}{1-\beta}\log\frac{(\omega_{KE}+\sddbar(\theta+\psi'_{1,\theta}+(1-\beta)\xi_{\beta,\theta}))^n}{(\omega_{KE}+\sddbar(\theta+\psi'_{1,\theta}))^n}\right)-\lambda\theta=0
\end{equation}
Denote by $\Gamma(\beta,\theta)$ the term on the left side, Then
\[
\Gamma(1,0)=0, \quad \Gamma(1,\theta)=P_0(\Delta_\theta\xi_{1,\theta})-\lambda\theta.
\]
Let $\theta(t)=t\theta\in\Lambda_1=Ker(\Delta_0+1)$. For any $\theta'\in\Lambda_1$,
\begin{eqnarray*}
\int_X\left.\frac{d}{dt}\Gamma(1,\theta)\right|_{t=0}\theta'\omega_{KE}^n/n!&=&-\lambda\int_X\theta\theta'\omega_{KE}^n/n!+\int_X(\dot{\Delta_\theta}\xi_{1,0}
+\Delta_0\dot{\xi}_{1,0})\theta'\omega_{KE}^n/n!\\
&=&-\lambda\int_X\theta\theta'\omega_{KE}^n/n!+\int_X-\theta^{i\bar{j}}(\xi_{1,0})_{i\bar{j}}\theta'\omega_{KE}^n/n!.
\end{eqnarray*}
Let $\xi=\xi_{1,0}\in \Lambda_1^{\perp}$. As the calculation in \cite{BM87}, we have
\begin{eqnarray*}
\int_X\theta^{i\bar{j}}\xi_{i\bar{j}}\theta'\omega_{KE}^n/n!&=&-\int_X(\theta^i\xi\indices{_{i\bar{j}}^{\bar{j}}}\theta'+\theta^i\xi_{i\bar{j}}\theta'^{\bar{j}})\omega_{KE}^n/n!=-\int_X(\theta^i\xi\indices{_{\bar{j}}^{\bar{j}}_{i}}\theta'-\theta^i\xi_{i}\theta'^{\bar{j}}_{\phantom{\bar{j}}\bar{j}})\omega_{KE}^n/n!\\
&=&-\int_X\theta^i((\Delta+1)\xi)_i\theta'\omega_{KE}^n/n!=\int_X(-\theta\theta'+\theta^i\theta'_i)(\Delta+1)\xi\omega_{KE}^n/n!\\
&=&-\int_X(\theta\theta'-\theta^i\theta'_i)(\lambda\phi-\log|s|_h^2)\omega_{KE}^n/n!
\end{eqnarray*}
In the last identiy, we used the relation in \eqref{xi10eq}. So
\begin{eqnarray*}
D_2\Gamma(1,0)(\theta)\theta'&=&-\lambda\int_X\theta\theta'\omega_{KE}^n/n!+\int_X(\theta\theta'-\theta^i\theta'_i)(\lambda\phi-\log|s|_h^2)\omega_{KE}^n/n!\\
&=&-Hess\mathscr{F}(\theta,\theta'). \quad (\mbox{ by equation }\eqref{hess2})
\end{eqnarray*}

By Lemma \ref{nondeg}, $D_2\Gamma(1,0)$ is invertible, so by
implicit function theorem, \eqref{projection2b} is solvable for
$\beta$ close to $1$.  So we get conical K\"{a}hler-Einstein metrics
$\omega_\beta$ for $\beta$ close to $1$ and by continuity
$\phi_\beta$ converges to $\phi_{KE}^D$ as $\beta\rightarrow 1$.
\end{proof}
\begin{rem}
As in \cite{BM87}, we can continue to calculate:
\begin{eqnarray*}
\int_X\theta^{i\bar{j}}\xi_{i\bar{j}}\theta'\omega_{KE}^n/n!&=&-\int_X(\theta\theta'-\theta^i\theta'_i)(\lambda\phi-\log|s|_h^2)\omega_{KE}^n/n!\\
&=&\frac{1}{2}\int_X(\Delta\theta\theta'+\theta\Delta\theta'+\theta^i\theta'_i+\theta_i\theta'^i)(\Delta+1)\xi\omega_{KE}^n/n!\\
&=&\frac{1}{2}\int_X(\Delta(\theta\theta'))(\Delta+1)\xi\omega_{KE}^n/n!=\frac{1}{2}\int_X\theta\theta'\Delta(\Delta+1)\xi\omega_{KE}^n/n!\\
&=&\frac{1}{2}\int_X\theta\theta'\Delta(\lambda\phi_1-\log|s|_h^2)\omega_{KE}^n/n!\\
&=&\frac{1}{2}\int_X\theta\theta'n\sddbar(\lambda\phi_1-\log|s|_h^2)\wedge\omega_{KE}^{n-1}/n!\\
&=&\frac{1}{2}n\lambda\int_X\theta\theta'\omega_{KE}^n/n!-\frac{1}{2}\int_{2\pi
D}\theta\theta'\omega_{KE}^{n-1}/(n-1)!.
\end{eqnarray*}
so that
\[
D_2\Gamma(1,0)(\theta)\theta'=-\lambda(1+n/2)\int_X\theta\theta'\omega_{KE}^n/n!+\frac{1}{2}\int_{2\pi
D}\theta\theta'\omega_{KE}^{n-1}/(n-1)!.
\]
However, it seems not straightforward to see that $D_2\Gamma(1,0)$
is nondegenerate using this formula.
\end{rem}
\begin{rem}
 One reason why we packed all the conical
spaces together in the space of admissible functions is because that
we need to work in different function space corresponding to
different cone angles. Strictly speaking, there are subtleties in
applying implicit functional theorem in this setting. However, we
expect one can generalize Donaldson's argument to validate the
application of implicit function theorem.
\end{rem}

\section{Relations to Song-Wang's work}\label{relSW}

In this section, we will briefly explain Song-Wang's results and
derive one of its implications.

For one thing, they also observe the interpolation property for the
log-Ding-energy. Secondly, their prominent idea of considering
pluri-anticanonical sections corresponds to the $\lambda\ge 1$ case
in our paper. Recall that $R(X)$ in the introduction (see
\eqref{R(X)}) is defined to be the greatest lower bound of Ricci
curvature of smooth K\"{a}hler metrics in $c_1(X)$. $R(X)$ was
studied in (
\cite{Ti92},\cite{Sze},\cite{Li11a},\cite{Li12},\cite{SW}).
Donaldson (\cite{Do11}) made the following conjecture
\begin{conj}[Donaldson]
Let $D\in |-K_X|$ be a smooth divisor, then there exists a conical
K\"{a}hler-Einstein metric on $(X,(1-\beta)D)$ if and only if
$\beta\in (0,R(X))$.
\end{conj}
Song-Wang proved a weak version of Donaldson's conjecture by
allowing pluri-anticanonical divisor and its dependence on $\beta$.
Translating their result in our notations, they proved
\begin{thm}[Song-Wang,\cite{SW}]
For any $\gamma\in (0,R(X))$ there exists a large
$\lambda\in\mathbb{Z}$ and a smooth divisor $D\in |\lambda
K_X^{-1}|$ such that there exists a conical K\"{a}hler-Einstein metric
on $(X, \lambda^{-1}(1-\gamma)D)$.
\end{thm}
\begin{rem}
Note that in general, $\lambda$ and $D$ may depend on $\gamma$.
$\gamma$ is related to the cone angle parameter $\beta$ by the
relation $\lambda^{-1}(1-\gamma)=1-\beta$ or equivalently,
$\gamma=r(\beta)=1-\lambda(1-\beta)$.
\end{rem}
The proof of this theorem can be
explained through the H\"{o}lder's inequality
\[
\int_X
e^{h_\omega-\gamma\phi}\frac{\omega^n}{n!|s|^{2(1-\gamma)/\lambda}}
\le \left(\int_X
e^{p(h_\omega-\gamma\phi)}\omega^n/n!\right)^{1/p}\left(\int_X
|s|^{-2q(1-\gamma)/\lambda}\omega^n/n!\right)^{1/q}
\]
where $p^{-1}+q^{-1}=1$. To make contact with the invariant $R(X)$,
one choose $p=\frac{t}{\gamma}$ for any $t\in (\gamma,R(X))$. (This
is related to the characterization of $R(X)$ through the properness
of twisted Ding-energy as in \cite{Li12}) Then
$q=(1-p^{-1})^{-1}=\frac{t}{t-\beta}$. Now the integrability of the
second integral on the right gives the restriction on $\lambda$:
$\frac{2q(1-\gamma)}{\lambda}-1<1$. This gives the the lower bound
of $\lambda$ in Song-Wang's theorem.
\[
\lambda> (1-\gamma)\frac{R(X)}{R(X)-\gamma}.
\]

One other important result Song-Wang proved is the construction of
toric conical K\"{a}hler-Einstein metrics can be combined with the
strategy in our paper to prove a version of Donaldson's conjecture
on toric Fano manifolds. We will explain this briefly.

Any toric Fano manifold $X_\triangle$ is determined by a reflexive
lattice polytope $\triangle\subset\mathbb{R}^n$ containing only $O$
as the interior lattice point. For any $P\in\mathbb{R}^n$, $P$
determines a toric $\mathbb{R}$-divisor $D_P\sim_{\mathbb{R}}-K_X$.
More concretly, assume that the polytope is defined by the
inequalities $ l_j(x)=\langle x,\nu_j\rangle+a_j\ge 0$. Then
$D_P=\sum_{j}l_j(P)D_j$. If $P\in\overline{\triangle}$ is a rational
point, then for any integer $\lambda$ such that $\lambda P$ is an
integral lattice point, there exist a genuine holomorphic section
$s^{\lambda}_P$ of $-\lambda K_X$ and an integral divisor $\lambda
D_P$.

Let $P_c$ be the barycenter of $\triangle$, then the ray
$\overrightarrow{P_cO}$ intersect the boundary $\partial\triangle$
at a unique point $Q$. Note that in general, $Q$ is a rational
point. In \cite{Li11a}, the first author proved $R(X)$ is given by
\begin{equation}\label{RX}
R(X)=\frac{|\overline{OQ}|}{|\overline{P_cQ}|}.
\end{equation}
For any $\gamma\in [0,1]$, define
$P_\gamma=-\frac{\gamma}{1-\gamma}P_c$. Then
$P_\gamma\in\overline{\triangle}$ if and only if $\gamma\in
[0,R(X)]$, which is also equivalent to $D_{P_\gamma}$ being
effective. In particular, $P_{R(X)}=Q$. Using these notations,
Song-Wang proved the following theorem by adapting the method in
Wang-Zhu's work (\cite{WZ04}) on the existence of K\"{a}hler-Ricci
solitons on toric Fano manifolds.
\begin{thm}[Song-Wang,\cite{SW}]
For any $\gamma\in [0,1]$, there exists toric solution to the
following equation:
\[
Ric(\omega)=\gamma\omega+(1-\gamma)\{D_{P_\gamma}\}.
\]
When $\gamma\in [0,R(X)]$ is rational, then the solution
$\omega_\gamma$ is a conical K\"{a}hler-Einstein metric on
$(X,(1-\gamma)D_{P_\gamma})$. In particular, when $\gamma=R(X)$,
there exists a conical K\"{a}hler-Einstein metric on $(X,
(1-R(X))D_{Q})$.
\end{thm}
\begin{exmp}\label{P(1,1,4)}
The above theorem can be generalized to toric orbifold case. (See
\cite{YaZh} for related work where the K\"{a}hler-Ricci soliton on
toric Fano orbifolds was considered) We will illustrate this by
showing the conical K\"{a}hler-Einstein on $X=\mathbb{P}(1,1,4)$
considered in section \ref{P2case} in the toric language. The
polytope determining $(X,-K_X)$ is the following rational polytope
$\triangle$.
\begin{figure}[h]
\begin{center}
\vspace*{-25mm}

\setlength{\unitlength}{0.8mm}
\begin{picture}(50,50)
\put(-18, 0){\line(1,0){56}} \put(0, -18){\line(0,1){30}}

{\thicklines\put(-8,-8){\line(1,0){48}}} \put(-8,-8){\line(0,1){12}}
\put(-8,4){\line(4,-1){48}}
\multiput(-16,-16)(8,0){9}{\multiput(0,0)(0,8){4}{\circle*{1}}}
\put(8,-4){\circle*{1.5}} \put(10,-6){$P_c$}
\put(8,-4){\line(-2,1){16}} \put(-14,4){$Q$}
\put(-8,4){\circle*{1.5}} \put(0,0){\circle*{1.5}} \put(-4,-4){$O$}
\put(-12,-12){$A$} \put(41,-12){$B$}
\end{picture}
\end{center}
\vspace*{10mm}
\end{figure}
 Note that $-2K_X$ is Cartier because $2\triangle$ is
a lattice polytope. $Q=(-1,1/2)$, $P_c=(1,-1/2)$. So
$R(X)=|\overline{OQ}|/|\overline{P_cQ}|=1/2$. $D_Q=3/2 D$, where the
divisor $D$ corresponds to the facet $\overline{AB}$. The conical
K\"{a}hler-Einstein satisfies the equation:
\[
Ric(\omega)=\frac{1}{2}\omega+(1-\frac{1}{2})\cdot\frac{3}{2}D.
\]
So the cone angle along $D$ is $2\pi\beta$ with $\beta=1-3/4=1/4$.
\end{exmp}
Now we show that Song-Wang's nice existence result implies Theorem
\ref{Doconj}.

\begin{proof}[Proof of Theorem \ref{Doconj}]
Let $\mathcal{F}_Q$ be the minimal face of $\triangle$ containing
$Q$. For any $\lambda\in\mathbb{Z}$ such that $\lambda Q$ is an
integral point, define a set of rational points by
\[
\mathcal{R}(Q,\lambda)=\{Q\}\bigcup \left(
(\overline{\triangle}\;\backslash\;\overline{\mathcal{F}_Q})\bigcap
\frac{1}{\lambda}\mathbb{Z}^n\right).
\]

Then we define the linear system $\mathscr{L}_{\lambda}$ to be the
linear subspace spanned by the holomorphic sections corresponding to
rational points in $\mathcal{R}(Q,\lambda)$:
\[
\mathscr{L}_{\lambda}=Span_{\mathbb{C}}\left\{s_P^{\lambda};
P\in\mathcal{R}(Q,\lambda)\right\} .
\]
Choose any general element $D\in\mathscr{L}_{\lambda}$, the
coefficient of the term $s_Q^{\lambda}$ is nonzero. Because $Q$ is a
vertex of the convex hull of $\mathcal{R}(Q,\lambda)$, there exists
a $\mathbb{C}^*$ action denoted by $\sigma(t)$ contained in the
torus action, such that
\[
\lim_{t\rightarrow 0} \sigma(t)^* D=\lambda D_Q.
\]
In this way, we construct a degeneration
$(\mathcal{X},\frac{1-R(X)}{\lambda}\mathcal{D},K_X^{-1})$ with
$\mathcal{X}=X\times\mathbb{C}$ and $\mathcal{D}_t=\sigma_t^* D$. By
Song-Wang's theorem in \cite{SW}, the central fibre
$(\mX_0,\frac{1-R(X)}{\lambda}(\lambda D_Q))=(X,(1-R(X))D_Q)$ has
conical K\"{a}hler-Einstein metric. So we use Theorem
\ref{specialdeg} to get the lower bound of
$\mathcal{M}_{X,\frac{1-R(X)}{\lambda} D}$. (As explained in the
proof of Theorem \ref{specialdeg}, in the present case, since
$\mathcal{X}=X\times\mathbb{C}$, we just need to use the trivial
geodesic and apply Berndtsson's result in \cite{Bern3} to get the
subharmonicity and complete the proof) On the other hand, because
$\lambda\ge 1$, we can use the interpolation result in Proposition
\ref{interpolate} to see that
$\mathcal{M}_{X,\frac{1-\gamma}{\lambda }D}$ is proper for any
$\gamma\in (0,R(X))$ (actually for any $\gamma\in (1-\lambda,R(X))$.
So there exists a conical K\"{a}ler-Einstein metric on
$(X,\frac{1-\gamma}{\lambda} D)$ for any $\gamma\in (0,R(X))$. There
can not be conical K\"{a}hler-Einstein metric for $\gamma\in
(R(X),1)$ is easy to get because the twisted energy is bounded from
below by the log-Ding-energy. For details, see \cite{SW} and also
\cite{Li12}. The non-existence for $\gamma=R(X)$ is implied by
Donaldson's openness theorem in \cite{Do11} (see Theorem
\ref{DoIFT}), since otherwise there exists conical
K\"{a}hler-Einstein for some $\gamma\in (R(X),1)$.
\end{proof}
\begin{rem}\label{smooth}
The smoothness of the generic member seems to be more subtle than we
first thought. We will discuss this a little bit using standard
toric geometry. For this, we first denote $\{H_i\}_{i=1}^N$ to be
the set of codimensional 1 face (i.e. facet) of $\triangle$. Define
\[
\mathcal{B}(\mathcal{F}_Q)=\left(\bigcup_{\mathcal{F}_Q\not\subset
H_i}H_i\right)\bigcap \mathcal{F}_Q.
\]
Now it's easy to see that the base locus of $\mathscr{L}_{\lambda}$
is equal to
\[
\mathbb{B}_Q=\bigcup_{\sigma\subset\mathcal{B}(\mathcal{F}_Q)}X_\sigma.
\]
Here for any face of $\triangle$ we denote $X_\sigma$ to be the
toric subvariety determined by $\sigma$. Indeed, this follows from
the following fact: if $P$ is any lattice point and $\mathcal{F}_P$
is the minimal face containing $P$. Define
\[
\textbf{Star}(\mathcal{F}_P)=\bigcup_{\mathcal{F}_P\subset
\sigma}\sigma.
\]
where $\sigma$ ranges over all the (closed) faces of $\triangle$.
(including $\triangle$ itself). Then the zero set of the
corresponding holomorphic section $s_P$ is the toric divisor
corresponding to the set
\[
\overline{\triangle}\;\backslash\left(\textbf{Star}(\mathcal{F}_P)\right)^{\circ}=\bigcup_{\mathcal{F}_P\not\subset
H_i}H_i\subset\partial\triangle.
\]
By Bertini's Theorem (\cite{GH}), the generic element
$D\in\mathscr{L}_{\lambda}$ is smooth away from $\mathbb{B}_Q$. To
analyze the situation near $\mathbb{B}_Q$, fix any vertex $P$ of
$\mathcal{F}$. We can choose integral affine coordinates
$\{x_i\}_{i=1}^n$ such that
\[
\mathcal{F}_Q=\bigcap_{i=m+1}^n\{x_i=0\}.
\]
We can also write $Q=\lambda(d_1,\dots,d_m,0,\dots,0)$ with $\lambda
d_i$ being positive integers. On the other hand, by standard toric
geometry, the normal fan of $\triangle$ at $P$ determines an affine
chart $\mathcal{U}_P$ on $X$. There exists complex coordinate
$\{z_i\}_{i=1}^n$ such that
$X_{\mathcal{F}_Q}\cap\mathcal{U}_P=\{z_{m+1}=0,\dots,z_n=0\}$.
Locally, the generic member $D$ in $\mathscr{L}_\lambda$ is given by
the equation of the form:
\[
\prod_{i=1}^m a_i z_i^{\lambda d_i}+\sum_{j=m+1}^n b_j z_j(1+
f_j(z_1,\dots,z_m))+\sum_{j,k=m+1}^n c_{jk} z_j z_k
g(z_1,\dots,z_n)).
\]
where $a_i,b_j\neq 0$. If we delete the lattice points corresponding
to terms $z_j f_j(z_1,\dots,z_m)$, then $C$ would be smooth near
$X_{\mathcal{F}_Q}\cap\mathcal{U}_P$. Since $\mathbb{B}_Q\subset
X_{\mathcal{F}_Q}$ and $\mathcal{U}_P$ covers $X_{\mathcal{F}_Q}$ as
$P$ ranges over all the vertices of $\mathcal{F}_Q$ we conclude that
$D$ is smooth at points in $\mathbb{B}_Q$ as well. This certainly
puts a lot of restriction on the sub-linear system. However, even if
we don't delete these lattice points, the generic member in
$\mathscr{L}_\lambda$ could be smooth. For example, this is the case
when $\mathcal{F}_Q$ has dimension $\le 1$ in which case the base
locus consists of isolated points. In particular, this is true when
the toric variety has dimension $\le 2$.
\end{rem}
\begin{rem}
The degeneration behavior in the toric case is closely related to
the study of degenerations in \cite{Li11b} where the current
$D_{P_\gamma}$ is replaced by $(1-\gamma)\omega$ with $\omega$ being
a smooth reference metric.
\end{rem}
\begin{exmp}
Let $X=Bl_p\mathbb{P}^2$. Let $[Z_0,Z_1,Z_2]$ be homogeneous
coordinate on $\mathbb{P}^2$. We can assume
$p=(1,0,0)\in\mathbb{C}^2=\{Z_0\neq 0\}\subset\mathbb{P}^2$. Let
$\pi: X\rightarrow\mathbb{P}^2$ be the blow down of exceptional
divisor $E$. For simplicity we use $H$ to denote both the hyperplane
class on $\mathbb{P}^2$ and its pull-back on $X$. Then $-K_X=3H-E$
and $-2K_X=6H-2E$. So divisors in $|-2K_X|$ correspond to the sextic
curves on $\mathbb{P}^2$ whose vanishing order at $0$ is at least 2.
More precisely, if $C$ is such a curve representing $6H$, then the
corresponding divisor $D(C)$ in $|-2K_X|=|6H-2E|$ is the strict
transform of $C$. In toric language $X$ is determined by the
following polytope:
\begin{figure}[h]
\begin{center}
%\vspace*{-25mm}
\begin{picture}(300,50)
%\begin{figure}[h]
 %\begin{center}\label{figure1}
 \put(50,0){
\setlength{\unitlength}{0.8mm}
\begin{picture}(50,50)
\put(-18, 0){\line(1,0){40}} \put(0, -18){\line(0,1){40}}
\put(-12,0){\line(0,1){24}} \put(-12,0){\line(1,-1){12}}
\put(-12,24){\line(1,-1){36}} \put(0,-12){\line(1,0){24}}
\put(-6,-6){\line(1,1){7}} \put(0,0){\circle*{1}}
\put(1,1){\circle*{1}} \put(-6,-6){\circle*{1}} \put(-10,-10){$Q$}
\put(-3,1){$O$} \put(2,2){$P_c$} \put(-12,0){\circle*{1}}
\put(-17,2){$B_1$} \put(0,-16){$B_2$} \put(0,-12){\circle*{1}}
\end{picture}}
%\end{center}
%\end{figure}
%\vspace{15mm}
%\end{center}
%\begin{center}
%\vspace*{-25mm}
%\begin{figure}[h]
% \begin{center}\label{figure1}
\put(200,0){ \setlength{\unitlength}{0.8mm}
\begin{picture}(50,50)
\thicklines{ \put(-18, 0){\line(1,0){40}} \put(0,
-18){\line(0,1){40}} \put(-12,0){\line(0,1){24}}
\put(-12,0){\line(1,-1){12}} \put(-12,24){\line(1,-1){36}}
\put(0,-12){\line(1,0){24}} } \put(0,0){\circle*{2}}
\put(-6,-6){\circle*{2}} \put(-10,-10){$Q$} \put(-12,0){\circle{2}}
\put(0,-12){\circle{2}}

\multiput(-12,24)(6,-6){4}{\multiput(0,0)(0,-6){4}{\circle*{2}}}
\multiput(24,-12)(-6,6){4}{\multiput(0,0)(-6,0){4}{\circle*{2}}}

{\thinlines \put(-6,-6){\line(-1,2){6}} \put(-6,-6){\line(2,-1){12}}
}

%\multiput(-6,-6)(-0.5,1){12}{\circle*{0.6}}
%\multiput(-6,-6)(1,-0.5){12}{\circle*{0.6}}

%\tikz
%\begin{tikzpicture}[scale=0.15]
%\draw [fill=lightgray](-12,6)--(-12,24)--(24,-12)--(6,-12)--(-12,6);
%\end{tikzpicture}
%\multiput(-6,18)(0,-6){4}{\circle*{1}}

%\put(-6,-6){\line(1,1){7}} \put(-3,1){$O$} \put(2,2){$P_c$} \put(1,1){\circle*{1}}{$O$}
\end{picture}
}

\end{picture}
\end{center}
\vspace{10mm}
\end{figure}

The invariant $R(X)=6/7$ was calculated in \cite{Sze} and
\cite{Li11a}. Since the point $Q=(-1/2,-1/2)$, it's easy to see that
$D_Q=1/2(F_1+F_2)+2 D_\infty$. By Song-Wang \cite{SW}, there is a
conical K\"{a}hler-Einstein metric on $(X,(1-R(X)) D_Q)=(X,1/7
D_Q)$.

Now $\lambda Q$ is integral when $\lambda$ is even. The generic
divisors in the linear system $\mathscr{L}_2$ correspond to the
sextic curves given by degree $6$ homogeneous polynomial of the form
\[
C: Z_0^4 Z_1Z_2+\sum_{i=0}^3\sum_{j=0}^{6-i}a_{ij}Z_0^i
Z_1^jZ_2^{6-i-j}=0.
\]
Let $\sigma_t$ be the $\mathbb{C}^*$-action given by
\[
(Z_0,Z_1,Z_2)\rightarrow (Z_0,t^{-1}Z_1,t^{-1}Z_2).
\]
Then $\lim_{t\rightarrow 0}\sigma_t\cdot C=\{Z_0^4 Z_1 Z_2=0\}$.
Equivalently, by taking strict transform, we get $\lim_{t\rightarrow
0}\sigma_t\cdot D(C)= 2D_Q$. The same argument applies to
$\lambda=2m$ being even, where the divisors in $\mathscr{L}_{2m}$
correspond to the degree $6m$ curves of the form:
\[
Z_0^{4m}Z_1^mZ_2^m+\sum_{i=0}^{4m-1}\sum_{j=0}^{6m-i}a_{ij}Z_0^{i}Z_1^jZ_2^{6m-i-j}=0.
\]
Note that the strict transform of such generic curves are smooth at
the base locus $\mathbb{B}_Q=B_1\cup B_2$ and so smooth everywhere.
\end{exmp}
\begin{rem}
From the above discussion, we see that when $\lambda$ is even, the
divisor degenerates while the ambient space stays the same. The case
when $\lambda=1$, or more generally when $\lambda$ is odd, is still
open. From the point of view in our strategy, the problem is that
the right degeneration to conical K\"{a}hler-Einstein pair is still
missing. In this case, we expect the degeneration also happens to
the ambient space, similar with the degree $2$ plane curve case
studied in section \ref{P2case}.
\end{rem}

\vspace*{7mm}

\noindent Chi Li\\ Department of Mathematics, Princeton University,
Princeton, NJ 08544, USA. \\E-mail: chil@math.princeton.edu

\vspace*{5mm}

\noindent Song Sun\\ Department of Mathematics, Imperial College,
London SW7 2AZ, U.K. \\E-mail: s.sun@imperial.ac.uk

\end{document}